\documentclass[psamsfonts]{amsart}

\usepackage{amssymb,amsfonts}
\usepackage[all,arc]{xy}
\usepackage{enumerate}
\usepackage{mathrsfs}
\usepackage{xcolor}
\usepackage[margin=1in]{geometry}
\usepackage{bbm}
\usepackage [english]{babel}
\usepackage [autostyle, english = american]{csquotes}
\MakeOuterQuote{"} 
\usepackage{tikz-cd} 
\usepackage{xpatch}
\usepackage{hyperref} 
\usepackage{amsrefs} 
\makeatletter
\renewcommand\th@plain{\slshape}
\xpatchcmd{\proof}{\itshape}{\slshape}{}{}
\makeatother

\newcommand{\norm}[1]{\left\lVert#1\right\rVert}

\newcommand{\ds}{\displaystyle}
\newcommand{\R}{\mathbb{R}}
\newcommand{\Z}{\mathbb{Z}}
\newcommand{\N}{\mathbb{N}}

\newcommand{\ttt}{\mathbf{t}}

\newcommand{\ASL}{${\rm ASL}_n(\mathbb{R})$}

\newcommand{\vv}{\mathbf{v}}
\newcommand{\ww}{\mathbf{w}}
\newcommand{\xx}{\mathbf{x}}
\newcommand{\yy}{\mathbf{y}}
\newcommand{\zz}{\mathbf{z}}

\newcommand{\dd}{\mathbf{d}}

\newcommand{\les}{\preccurlyeq}

\newcommand\da{Diophantine approximation}
\newcommand\ssm{\smallsetminus}

\newcommand{\dint}{\mathrm{d}}
\newcommand{\SL}{\operatorname{SL}_n(\R)}

\newcommand {\new}[1]   {\textcolor{blue}{#1}}
\newcommand {\comm}[1]   {\textcolor{red}{#1}}
\newcommand {\rmv}[1]   {\textcolor{brown}{{#1}}}
\newcommand\eq[2]{
\begin{equation}
\label{eq:#1}
{#2}
\end{equation}
}
\newcommand{\equ}[1]{\eqref{eq:#1}}
\newcommand{\ignore}[1]{}

\newcommand{\e}{\varepsilon}

\theoremstyle{plain}
\newtheorem{thm}{Theorem}[section]
\newtheorem{cor}[thm]{Corollary}
\newtheorem{rmk}[thm]{Remark}

\newtheorem{lem}[thm]{Lemma}

\newtheorem{manualtheoreminner}{Theorem}
\newenvironment{manualtheorem}[1]{%
  \renewcommand\themanualtheoreminner{#1}%
  \manualtheoreminner
}{\endmanualtheoreminner}

\theoremstyle{definition}
\newtheorem{defn}[thm]{Definition}

\newtheorem*{standing}{Standing Assumptions}

\theoremstyle{remark}

\numberwithin{equation}{section}

\title[Inhomogeneous Diophantine approximation for generic homogeneous functions]{Inhomogeneous Diophantine approximation \\ for generic homogeneous functions}    

\author{Dmitry Kleinbock}
\address{
\begin{itemize}
\item[] Dmitry Kleinbock
\item[] Department of Mathematics
\item[] Brandeis University 
\item[] Goldsmith Building $218$     
\item[] Waltham, MA $02454$\textendash$9110$
\item[] USA
\item[] \href{mailto:kleinboc@brandeis.edu}{{\tt kleinboc@brandeis.edu}} 
\item[] \href{https://orcid.org/0000-0002-9418-5020}{\tt https://orcid.org/0000-0002-9418-5020}\end{itemize}}

\author{Mishel Skenderi}    
\address{
\begin{itemize}
\item[] Mishel Skenderi
\item[] Department of Mathematics
\item[] The University of Utah
\item[] $155$ South $1400$ East JWB $233$
\item[] Salt Lake City, UT $84112$\textendash$0090$ 
\item[] USA
\item[] \href{mailto:mskenderi@math.utah.edu}{{\tt mskenderi@math.utah.edu}}     
\item[] \href{https://orcid.org/0000-0001-8409-1613
}{{\tt https://orcid.org/0000-0001-8409-1613
}}   \end{itemize}}     

\date{August 2022}
\thanks{The co-author Kleinbock has been supported by  NSF grant 
DMS-1900560}

\begin{document}

\begin{abstract} 
The present paper is a sequel to [Monatsh.~Math.\ {\bf 194} (2021), 523--554] in which results of that paper are generalized so that they hold in the setting of inhomogeneous Diophantine approximation. Given any integers $n \geq 2$ and $\ell \geq 1$, any ${\pmb \xi} = \left(\xi_1, \dots , \xi_\ell \right) \in \R^\ell$, and any homogeneous function \linebreak  $f = \left(f_1, \dots , f_\ell \right): \mathbb{R}^n \to \mathbb{R}^\ell$ that satisfies a certain nonsingularity assumption, we obtain a biconditional criterion on the approximating function $\psi = \left(\psi_1, \dots , \psi_\ell \right): \mathbb{R}_{\geq 0} \to \left(\mathbb{R}_{>0}\right)^\ell$ for a generic element $f \circ g$ in the $\operatorname{SL}_n(\mathbb{R})$-orbit of $f$ to be (respectively, not to be) $\psi$-approximable at ${\pmb \xi} = (\xi_1,\dots,\xi_n)$: that is, for there to exist infinitely many (respectively, only finitely many) $\mathbf{v} \in \mathbb{Z}^n$ such that $\left|\xi_j - \left( f_j \circ g\right)(\mathbf{v})\right| \leq  \psi_j(\|\mathbf{v}\|)$ for each $j \in \left\lbrace 1, \dots, \ell \right\rbrace$. In this setting, we also obtain a sufficient condition for uniform approximation. We also consider some examples of $f$ that do not satisfy our nonsingularity assumptions and prove similar results for these examples. Moreover, one can replace $\operatorname{SL}_n(\R)$ above by any closed subgroup of $\operatorname{ASL}_n(\mathbb{R})$ that satisfies certain integrability axioms (being of Siegel and Rogers type) introduced by the authors in the aforementioned previous paper.      
\end{abstract}  

\maketitle

\tableofcontents

\section{Introduction}\label{intro}
Let $f$ be an indefinite nondegenerate quadratic form in $n\geq 3$ real variables that is not a real multiple of a quadratic form with rational coefficients. The Oppenheim\textendash Davenport Conjecture, resolved affirmatively by Margulis \cite{Mar89}, states that every real number is an accumulation point of $f(\Z^n)$: \eq{density}{\text{For any } \xi \in \R \text{ and any } \varepsilon \in \R_{>0}, \text{ there exist infinitely many }\vv\in\Z^n\text{ for which } |f(\vv) - \xi| \leq \varepsilon.} The rich history of the Oppenheim\textendash Davenport conjecture and its seminal resolution by Margulis, among various other related topics, are extensively discussed in Margulis's survey \cite{Mar97}. The influence of Margulis's theorem and related problems continues unabated to this day. As of a few years ago, there has been a great increase of activity in proving \textit{effective} variants of Margulis's theorem for \textit{generic} quadratic forms and other homogeneous polynomials:~for instance, one often considers the $\operatorname{SL}_n(\R)$-orbit (under the natural action) of a real homogeneous polynomial in $n$ real variables; one then has a natural notion of measure class (and thus measure-theoretic genericity) for this orbit. Let us briefly recall some recent results that exemplify this circle of problems. Throughout this paper, we write $\N := \Z_{\geq 1}$:~that is to say, we do not consider $0 \in \Z$ to be a natural number.            

For any $\beta \in \R_{\geq 1}$ and any $(p, q) \in \N^2$ with $p+q = n \geq 3,$ let $\ds F_{p, q, \beta} : \R^n \to \R$ be given by \eq{generalquad}{F_{p, q, \beta}(\xx) := \sum_{j=1}^p  \left|x_j\right|^\beta - \sum_{k=p+1}^{n} \left|x_{k}\right|^\beta.} 

Generalizing earlier results of Ghosh\textendash Gorodnik\textendash Nevo \cite{GhoshGorodnikNevo} and Athreya\textendash Margulis \cite{Pol}, Kelmer\textendash Yu \cite{KY2} proved the following theorem.  

\begin{thm}[{\cite[Corollary 2]{KY2}}] \label{KY} Let $\beta \in 2\N$, let $n \in \Z_{> \beta}$, and let $p,q\in\N$ be such that $p+q = n.$ Let $s \in (0, n - \beta) \subset \R.$ Let $\xi \in \R.$ Let $\|\cdot\|$ be an arbitrary norm on $\R^n.$ Then for Haar-almost every $g \in \operatorname{SL}_n(\R)$ the following holds:~for each sufficiently large $T \in \R_{>0}$ there exists $\vv \in \Z^n$ with \[ 0 < \|\vv\| \leq T \quad \text{and} \quad \left|-\xi + F_{p, q, \beta}(g\vv) \right| \leq T^{-s}. \] \end{thm} \noindent  The proofs in \cite{GhoshGorodnikNevo} made use of representation theory and effective mean ergodic theorems, while those in \cite{Pol} and then \cite{KY2} employed comparatively elementary means (namely, first and second moment formulae for the Siegel transform in the geometry of numbers). For other results of a similar nature, see {\cite{EMM, MM, LM, Bourgain, GhoshKelmer, GKY_a, GKY_b,  Anishnew}}. The authors previously studied problems of this sort in \cite{KS}; the work in \cite{KS} was, however, limited to \textit{homogeneous} approximation, which corresponds to the special case $\xi = 0$ in the context of \equ{density}. The purpose of the present paper is to establish \textit{inhomogeneous} analogues of the results of \cite{KS}. In order both to recall the results of \cite{KS} and to present the results of the present paper, we proceed to establish some notation and definitions; much of the notation and many of the definitions that follow were first established in \cite{KS} and are recalled here for the convenience of the reader.

\smallskip

Now and hereafter, we shall denote by $n$ an arbitrary element of $\Z_{\geq 2}$ and by $\ell$ an arbitrary element of {$\N$}. Elements of $\R^n$ and of $\R^\ell$ shall always be regarded as column vectors, even though they may be written as row vectors for notational convenience. If $k \in \N$ 
and $E\subset \R^k,$  we define $E_{\neq 0} := E \ssm \left\lbrace\mathbf{0}_{\R^k}\right\rbrace.$ We say that 
$\vv = (v_1, \dots , v_n) \in \Z^n$ is \textsl{primitive} if $\gcd(v_1, \dots , v_n) = 1.$ We let $\Z^n_{\rm pr}$ denote the set of all primitive points of $\Z^n.$ We shall always denote the usual Lebesgue measure on any Euclidean space by $m$. Throughout this paper, we shall use the Vinogradov notation $\ll$ and use $\asymp$ to denote that both "$\ll$" and "$\gg$" hold; we shall attach subscripts to the symbols $\ll$ and $\asymp$ to indicate the parameters, if any, on which the implicit constants depend.         

\begin{defn}[{\cite[Definition 1.1]{KS}}]\label{order}
We define a non-strict partial order $\les$ on $\R^{\ell}$ as follows. For any $\ds \xx = (x_1, \dots , x_{\ell}) \in \R^{\ell}$ and any $\ds \yy = (y_1, \dots , y_{\ell}) \in \R^{\ell},$ we write $\xx \les \yy$ if and only if $x_j \leq y_j$ for each $j \in \{1, \dots , {\ell}\}$.  
\end{defn}

\begin{defn}[{\cite[Definitions 1.2, 3.3, 3.6, and 3.7]{KS}}]\label{basicapdefns}
Let \[f = (f_1, \dots , f_{\ell}) : \R^n \to \R^{\ell} \ \ \ \ \ \text{and} \ \ \ \ \  \psi = (\psi_1, \dots , \psi_{\ell}) : \R_{\geq 0} \to \left(\R_{>0}\right)^{{\ell}}\] be given, let $\nu$ be an arbitrary norm on $\R^n$, and let $\mathcal{P}$ be an arbitrary subset of $\Z^n.$ 
\smallskip
\begin{itemize}
\item  We abuse notation and write $|f|$ to denote the function $\ds (|f_1|, \dots , |f_\ell| ) :\R^n \to \R^{\ell}.$
\item  We define $\ds \mathcal{Z}(f) := f^{-1}\left( \mathbf{0}_{\R^\ell} \right)\smallsetminus\{\mathbf{0}_{\R^n} \}$.     
\item We define $\ds A_{f, \psi, \nu} := \left\{ \xx \in \R^n : |f(\xx)| \les \psi\big(\nu(\xx)\big)\right\}.$
\item For any $T \in \R_{>0}$ and any ${\pmb\varepsilon}\in \left(\R_{>0}\right)^\ell,$ we define $$B_{f, {\pmb\varepsilon}, \nu, T} := \left\{  \xx \in \R^n : |f(\xx)| \les {\pmb\varepsilon} \ \text{and} \ \nu(\xx) \leq T \right\}.$$ 
\item We  say that $f$ is $(\psi, \nu, \mathcal{P})$-\textsl{approximable} if $A_{f, \psi, \nu} \cap \mathcal{P}$ has infinite cardinality. 
\item We  say that $f$ is \textsl{uniformly} $(\psi, \nu, \mathcal{P})$-\textsl{approximable} if {$B_{f, \psi(T), \nu, T} \cap \mathcal{P}  \neq \varnothing$ for each sufficiently large $T \in \R_{> 0}$}.  
\ignore{\item Let $\ds t_\bullet = \left(t_k\right)_{k \in \N}$ be any strictly increasing sequence of elements of $\R_{>0}$ with $\ds \lim_{k \to \infty} t_k = \infty.$ We then say that $f$ is $\ds t_\bullet$-\textsl{uniformly} $\left( \psi, \nu, \mathcal{P} \right)$-approximable if {$B_{f, \psi(t_k), \nu, t_k} \cap \mathcal{P}  \neq \varnothing$ for each sufficiently large $k \in \N$}.     
\item Let $\ds t_\bullet = \left(t_k\right)_{k \in \N}$ be any strictly increasing sequence of elements of $\R_{>0}$ with $\ds \lim_{k \to \infty} t_k = \infty.$ We say that $t_\bullet$ is \textsl{quasi-geometric} if, in addition to the preceding, the set $\ds \left\lbrace t_{k+1}/t_k : k \in \N \right\rbrace$ is bounded. 
\item We say that  $f$ is \textsl{subhomogeneous} if $f$ is Borel measurable and there exists ${\delta} = {\delta}(f) \in \R_{>0}$ such that for each $t \in (0, 1) \subset \R$ and each $\xx \in \R^n$ we have $\ds |f(t\xx)| \les t^{\left({\delta}\right)} |f(\xx)|.$}       
\item We say that  $f$ is \textsl{homogeneous} if it is Borel measurable and there exists some $\dd = \dd(f) = (d_1,\dots,d_\ell) \in (\R_{>0})^\ell$ such that for each $t \in \R_{\geq 0}$, each $j \in \{1, \dots , \ell\},$ and each $\xx \in \R^n$ we have $\ds f_j(t\xx) = t^{d_j} f_j(\xx).$ We refer to $\dd = \dd(f)$ as the \textsl{degree of} $f.$\footnote{This terminology is slightly abusive because $\dd = \dd(f)$ is unique if and only if each component of $f$ is nonzero.}   
\item We say that $\psi$ is \textsl{regular} if it is Borel measurable and there exist real numbers $a = a(\psi) \in \R_{>1}$ and $b = b(\psi) \in \R_{>0}$ such that for each ${z} \in \R_{>0}$ one has $b\psi({z}) \les \psi(a {z}).$
\item We say that $\psi$ is \textsl{nonincreasing} if each component function of $\psi$ is nonincreasing in the usual sense.
\end{itemize}
\end{defn}

\begin{defn}\label{shiftedfunctiondef}
For any function $f: \R^n \to \R^\ell$ and any ${\pmb \xi} \in \R^\ell,$ define $\ds _{\pmb \xi} f: \R^n \to \R^\ell$ to be the function given by $\ds _{\pmb \xi} f(\xx) := -{\pmb \xi} + f(\xx).$ 
\end{defn}

\begin{defn}\label{logcorrectdefn}
For any 
$s \in \R_{>0}$  and any $\varepsilon \in \R_{\geq 0},$ let $\varphi_{s, \varepsilon} : \R_{\geq 0} \to \R_{>0}$ be any regular and nonincreasing function such that for every sufficiently large $t \in \R_{> 1}$ we have $\varphi_{s, \varepsilon}(t)= t^{-s} \left(\log t \right)^{\varepsilon}.$
\end{defn}

\ignore{{\begin{defn}\label{logcorrectdefn}
For any $j \in \N,$ let $\log^{{\circ j}}$ denote the $j$-fold iterated logarithm function. For any $i \in \N$, any $s \in \R_{>0}$, and any $\varepsilon \in \R_{> 0},$ let $\psi_{i, s, \varepsilon} : \R_{\geq 0} \to \R_{>0}$ be any regular and nonincreasing function such that for every sufficiently large $t \in \R_{> 1},$ we have \[ \psi_{i, s, \varepsilon}(t)= t^{-s} \cdot \left( \log^{{\circ(i+1)}}(t) \right)^{1 + \varepsilon} \cdot \prod_{j=1}^i \log^{{\circ j}}(t). \] For any $s \in \R_{>0},$ let $\psi_{s} : \R_{\geq 0} \to \R_{>0}$ be any regular and nonincreasing function such that for every $t \in \R_{> 1},$ we have $\psi_{s}(t) = t^{-s}.$ 
\end{defn}}}

Using our newly introduced terminology, we may now restate Theorem \ref{KY} as follows. (Note that statement (i) is a straightforward corollary of statement (ii).) 
\begin{manualtheorem}{1.1$'$}[{\cite[Corollary 2]{KY2}}] \label{KYredux}
Let $\beta \in 2\N$, let $n \in \Z_{> \beta}$, and let $p,q\in\N$ be such that $p+q = n.$ Let $s \in (0, n - \beta) \subset \R.$ Let $\xi \in \R.$ Let $\|\cdot\|$ be an arbitrary norm on $\R^n.$ Set $f := F_{p, q, \beta}.$ Then the following hold.
\begin{itemize}   
\item[(i)] The function $\left({_{\xi}} f\right) \circ g$ is $\left(\varphi_{s,0}, \|\cdot\|, \Z^n_{\neq 0}\right)$-approximable for Haar-almost every $g \in \SL$.   
\item[(ii)] The function $\left({_{\xi}} f\right) \circ g$ is uniformly $\left(\varphi_{s,0}, \|\cdot\|, \Z^n_{\neq 0}\right)$-approximable for Haar-almost every $g \in \SL$. 
\end{itemize}
\end{manualtheorem}

\smallskip

We now proceed to state certain results that will be formulated and proved in greater generality in \S \ref{acc}. Before we state the first result, let us recall an elementary notion from differential topology. 

\begin{defn}\label{manifoldregular} Let $M$ and $N$ be $\mathscr{C}^1$ manifolds that are Hausdorff, second-countable, and without boundary. Let $U$ be an open subset of $M.$ Let $f : M \to N$ be a function, and suppose that $f$ is continuously differentiable on $U.$ Let $x \in U$ be given. We say that $x$ is a \textsl{regular point of} $f$ if the map $\ds D_x f : T_x M \to T_{f(x)} N$ is surjective.  \end{defn}

\begin{thm}\label{simplesubmerapprox}
Let $\ds \psi = (\psi_1, \dots , \psi_{\ell}) : \R_{\geq 0} \to \left(\R_{>0}\right)^{\ell}$ be regular and nonincreasing. Let $\ds f = (f_1, \dots , f_{\ell}) 
: \R^n \to \R^{\ell}$ be homogeneous of degree  $\dd = \dd(f) = (d_1,\dots,d_\ell)\in (\R_{>0})^\ell$. Suppose further that $f$ is continuously differentiable on $\R^n_{\neq 0}$, that $\mathcal{Z}(f) \neq \varnothing$, and that 
\eq{submersion}{\text{each element of $\mathcal{Z}(f)$ is a regular point of } f.} Let $\nu$ be an arbitrary norm on $\R^n$. Let ${\pmb \xi} \in \R^\ell.$ Set $d := \sum_{j=1}^\ell d_j.$ Then the following hold. 
\begin{itemize}
\item[ \rm (i)] If $\ds \int_1^{\infty} t^{n-(d+1)} \left(\prod_{j=1}^\ell \psi_j(t) \right) \dint t$ is finite $($respectively, infinite$)$,
then $\left({_{\pmb \xi}} f\right) \circ g$ is $\left(\psi, \nu, \Z^n_{\neq 0}\right)$-approximable for Haar-almost no $($respectively, almost every$)$ $g \in \SL$. 
\item[\rm (ii)] Suppose that $d < n$ and that the infinite series $\ds \sum_{k=1}^{\infty}\left[ 2^{k(n-d)} \prod_{j=1}^\ell \psi_j\left(2^k\right) \right]^{-1}$ converges. Then $\left({_{\pmb \xi}} f\right) \circ g$ is {\sl uniformly} $\left(\psi, \nu, \Z^n_{\neq 0}\right)$-approximable for Haar-almost every $g \in \SL$.
\end{itemize}
\end{thm}

\begin{rmk} \rm 
Notice that no component of $f$ in the above theorem is required to be a polynomial. Notice also that neither the integral criterion in (i) nor the summatory condition in (ii) features any dependence on ${\pmb \xi} \in \R^\ell.$
\end{rmk}

\begin{rmk}\label{primasl} \rm    It will follow from a more general result in \S \ref{acc} that the preceding theorem remains true if one replaces   
\begin{itemize}
\item 
each instance of $\Z^n_{\neq 0}$ by $\Z^n_{\rm pr}$; 
\vspace{0.05in}
\item 
each instance of $\Z^n_{\neq 0}$ by $\Z^n$ and each instance of $\SL$ by $\operatorname{ASL}_n(\R)$ . 
Here, $\operatorname{ASL}(\R) := \operatorname{SL}_n(\R) \ltimes \R^n,$ the group of affine bijections $\R^n \to \R^n$ that preserve both volume and orientation. 
\end{itemize}   \end{rmk}

\begin{rmk}\label{examples} \rm Let $\beta$, $p$, $q$, $n = p+q$, and $\ds F_{p, q, \beta} : \R^n \to \R$ be as in \equ{generalquad}. Suppose further that $1 < \beta < n.$ Since $\beta>1$, the function $\ds F_{p, q, \beta}$ is continuously differentiable on $\R^n.$ Since $p \geq 1$ and $q \geq 1$, we have $\mathcal{Z}\left(F_{p, q, \beta}\right) \neq \varnothing.$ It is also easy to verify that each element of $\R^n_{\neq 0}$ is a regular point of $F_{p, q, \beta}.$ The hypotheses of Theorem \ref{simplesubmerapprox} are thus satisfied when $\ell = 1$, $d = d_1 = \beta$, and $\ds f = F_{p, q, \beta}$. This shows that statement (ii) of Theorem \ref{simplesubmerapprox} implies Theorem \ref{KYredux} (equivalently, Theorem \ref{KY}). Moreover, we can prove a generalization of Theorem \ref{KYredux} as follows. Let $\xi \in \R$ be arbitrary, let $\nu$ be an arbitrary norm on $\R^n,$ and set $f = F_{p, q, \beta}$ to simplify notation. The following then hold. 
\begin{itemize}
\item The function $\ds \left({_{ \xi}} f \right) \circ g$ is $\left(\varphi_{n-\beta,0}, \nu, \Z^n_{\rm pr}\right)$-approximable for Haar-almost every $g \in \SL$. This generalizes (i) of Theorem \ref{KYredux} to include the case of the critical exponent $s = n-\beta.$
\item The function $\ds \left({_{ \xi}} f \right) \circ g$ is uniformly $\left(\varphi_{n-\beta,\varepsilon}, \nu, \Z^n_{\rm pr}\right)$-approximable for every $\varepsilon \in \R_{>1}$ and Haar-almost every $g \in \SL$. This generalizes (ii) of Theorem \ref{KYredux}. In fact, one can generalize (ii) of Theorem \ref{KYredux} even further by suitably modifying the definition of $\varphi_{n-\beta,\varepsilon}$ to include an arbitrary  finite number of iterated logarithms.  \end{itemize}  \end{rmk}     

\begin{rmk} \rm 
Our general framework also allows for vector-valued examples of $f.$ For instance, let $f_1 = F_{p, q, \beta}:\R^n \to \R$ be as in Remark \ref{examples}, and let $f_2 : \R^n \to \R$ be an $\R$-linear transformation. (These functions may remind the reader of the setting in the papers \cite{Gor, Dani, BG}.) Then $f := \left(f_1, f_2\right) : \R^n \to \R^2$ satisfies the hypotheses of Theorem \ref{simplesubmerapprox} if and only if the intersection $\mathcal{Z}(f_1) \cap \mathcal{Z}(f_2)$ is nonempty and transverse. One thereby obtains a criterion for the asymptotic approximability of and a sufficient condition for the uniform approximability of almost every element in the $\SL$-orbit of ${_{ {\pmb \xi}}} f,$ for any ${\pmb \xi} \in \R^2.$ (The corresponding conull subsets of $\SL$ depend on ${\pmb \xi} \in \R^2.$)        
\end{rmk}

We are also able to obtain results akin to Theorem \ref{simplesubmerapprox} in certain special cases wherein the nonsingularity condition \equ{submersion} does not hold. In Theorem \ref{simplesubmerapprox} and in all the aforementioned examples, the integral and summatory conditions obtained were all independent of ${\pmb \xi} \in \R^\ell.$ As the following two examples illustrate, this independence need no longer hold when the nonsingularity condition \equ{submersion} fails to hold.

\begin{thm}\label{simpleproduct}
Let $\psi : \R_{\geq 0} \to \R_{>0}$ be  regular and nonincreasing. Let ${\omega} \in \R_{>0}$ be arbitrary. Let $f : \R^n \to \R$ be given by $f(\xx) := \left(\prod_{i=1}^n \left| x_i \right| \right)^{\omega}.$ Let $\nu$ be an arbitrary norm on $\R^n.$ Let $\xi \in \R_{\geq 0}.$ Then the following hold.   
\vspace{0.05in}
\begin{itemize}
\item[ \rm (i)]  If \[ \begin{cases}
 { \ds \int_1^{\infty} \frac{\psi(t)^{1/{\omega}}}{t} \left(\log t\right)^{n-2} \, \dint t }  &  \text{ if } \xi = 0  \\ { \ds \int_1^{\infty} \frac{\psi(t)}{t} \left(\log t\right)^{n-2} \, \dint t }  &  \text{ if } \xi > 0 \end{cases} \] is finite $($respectively, infinite$)$,
then $\left({_{\xi}} f\right) \circ g$ is $\left(\psi, \nu, \Z^n_{\neq 0}\right)$-approximable for Haar-almost no $($respectively, almost every$)$ $g \in \SL$. \vspace{0.05in}   
\item[\rm (ii)]  If the infinite series \[ \begin{cases}
{ \ds \sum_{k=1}^{\infty}\left[ k^{n-1} \psi\left(2^k\right)^{1/{\omega}} \right]^{-1} }   &  \text{ if } \xi = 0  \\ { \ds \sum_{k=1}^{\infty}\left[ k^{n-1} \psi\left(2^k\right) \right]^{-1} }  &  \text{ if } \xi > 0 \end{cases} \]  converges, then $\left({_{\xi}} f\right) \circ g$ is {\sl uniformly} $\left(\psi, \nu, \Z^n_{\neq 0}\right)$-approximable for Haar-almost every $g \in \SL$.
\end{itemize}
\end{thm}

\smallskip

\begin{thm}\label{simplemax}
Let $\psi : \R_{\geq 0} \to \R_{>0}$ be  regular and nonincreasing. Let $p \in \{1, \dots , n - 1 \}$ and $\mathbf{z} = (z_1, \dots , z_p) \in \left(\R_{>0}\right)^p$ be given. Let $f: \R^n \to \R$ be given by \[ f(x_1, \dots , x_n) := \max\left\lbrace |x_i|^{z_i}: 1 \leq i \leq p\right\rbrace. \] Set $z := \sum_{i=1}^p  z_i^{-1}.$ Let $\nu$ be an arbitrary norm on $\R^n.$ Let $\xi \in \R_{\geq 0}.$ Then the following hold.   
\vspace{0.05in}       
\begin{itemize}
\item[ \rm (i)] If \[ \begin{cases}
 { \ds \int_1^{\infty} \psi(t)^z  {t}^{n - (p+1)} \ \dint t  }  &  \text{ if } \xi = 0  \\ { \ds \int_1^{\infty} \psi(t) \, {t}^{n - (p+1)} \ \dint t  }  &  \text{ if } \xi > 0 \end{cases} \] is finite $($respectively, infinite$)$,
then $\left({_{\xi}} f\right) \circ g$ is $\left(\psi, \nu, \Z^n_{\neq 0}\right)$-approximable for Haar-almost no $($respectively, almost every$)$ $g \in \SL$. 
\vspace{0.05in}
\item[\rm (ii)] If the infinite series \[ \begin{cases}
{ \ds \sum_{k = 1}^\infty 
\left[ 2^{(n-p)k} \psi\left(2^k\right) ^z \right]^{-1}}   &  \text{ if } \xi = 0  \\ { \ds \sum_{k = 1}^\infty 
\left[ 2^{(n-p)k} \, \psi\left(2^k\right) \right]^{-1}}  &  \text{ if } \xi > 0 \end{cases} \] converges, then $\left({_{\xi}} f\right) \circ g$ is {\sl uniformly} $\left(\psi, \nu, \Z^n_{\neq 0}\right)$-approximable for Haar-almost every $g \in \SL$.
\end{itemize}
\end{thm}  

\ignore{  \begin{rmk} \rm 
\begin{itemize}
\item[]
\item[(i)] Let $q$ be any element of $\mathbb{Q}$ for which there exist $q_1, q_2 \in \N$ such that $q = q_1/q_2$, $\gcd(q_1, q_2) = 1$, and $\gcd(2, q_2) = 1.$ Let $f : \R^n \to \R$ be given by $f(\xx) := \left(\prod_{j=1}^n  x_j \right)^q.$

\end{itemize}

\end{rmk} }    

\begin{rmk}\label{singularprimasl} \rm
The assertions in Remark \ref{primasl} apply to Theorems \ref{simpleproduct} 
and \ref{simplemax} as well. As in Remark \ref{examples}, one may easily deduce corollaries of the preceding two theorems for $\psi$ of the form $\varphi_{s,\varepsilon}.$     \end{rmk}

\noindent Theorems \ref{simplesubmerapprox}, \ref{simpleproduct}, and \ref{simplemax} are immediate special cases of Theorems \ref{submerapprox}, \ref{productcoord}, and \ref{khintchine}; the latter three theorems are consequences (albeit not immediate ones) of Theorems \ref{asymptoticmain} and \ref{mainuniformresult}. (Likewise, the assertions in Remark \ref{primasl} are immediate consequences of Theorem \ref{submerapprox}.) Results similar to the aforementioned ones were established by the authors in \cite{KS} for the special case ${\pmb \xi } = \mathbf{0}_{\R^\ell}$: see \cite[Theorems 3.4 and 3.8]{KS}. 

\begin{rmk}\label{simplezerovsnonzero} \rm  
From the preceding two theorems, we see that the failure of $f : \R^n \to \R$ to satisfy the nonsingularity condition \equ{submersion} of Theorem \ref{simplesubmerapprox} is necessary, but not sufficient, to ensure dependence on $\xi.$ This is an interesting phenomenon that is worthy of further investigation. It would also be interesting to understand whether there exist reasonably well-behaved $f: \R^n \to \R^\ell$ that exhibit dependence on ${\pmb \xi}$ finer than the rather crude "zero or nonzero" dependence.    \end{rmk}

\ignore{\comm{Also two more suggestions:
\begin{itemize}
\item we spend so much time in the proof dealing with the ASL case, that it is strange that we never mention what it actually means. Maybe explain it somewhere?
\item Same about $\ell > 1$. Maybe have an example where $f$ is a pair of quadratic forms, or a pair (quadratic, linear)?
\end{itemize} }

\comm{MISHEL: I tried to address the first bullet point above in Remarks \ref{primasl} and \ref{singularprimasl}. Now that the submersion approximation theorem is much simpler (no restriction to the sphere), does that address the suggestion in the second bullet point? Or should we say something like the gradients should be nonzero and not scalar multiples of each other for $\ell=2.$ But I think that's pretty obvious that it's an easily checkable criterion for $\ell=2.$ Finally, I thought it might be good to put the definition of subhomogeneity between Lemmata 2.1 and 2.2, so that's what I ended up doing.  }}

\section{General results}\label{gen}

{We proceed by laying the groundwork necessary to state and prove our general results. We follow the presentation given by the authors in \cite{KS}. }

\smallskip

Let $\operatorname{ASL}_n(\R) := \operatorname{SL}_n(\R) \ltimes \R^n.$ We assume that $\operatorname{ASL}_n(\R)$ and each subgroup thereof act on $\R^n$ in the usual manner: for any $\langle h, \mathbf{z} \rangle \in \operatorname{ASL}_n(\R)$ and any $\xx \in \R^n,$ we have $\langle h, \mathbf{z} \rangle \xx = \mathbf{z} + h\xx$. Let $G$ be an arbitrary closed subgroup of $\operatorname{ASL}_n(\R)$, and let $\mathcal{P}$ be an arbitrary subset of $\Z^n.$ Define $\Gamma = \Gamma(G, \mathcal{P})$ to be the subgroup of $G$ given by \eq{gamma}{ \Gamma := \{ g \in G : g\mathcal{P} = \mathcal{P} \}.} Now and hereafter, we assume that $\Gamma$ is a lattice in $G$. (For each example of $(G, \mathcal{P})$ that we shall consider, the group $\Gamma$ will indeed be a lattice in $G.$) Set $X := G/\Gamma.$ We let $\mu_G$ denote the Haar measure on the unimodular group $G$ that is normalized so that $\operatorname{vol}\left(G/\Gamma\right) = 1.$ We then let $\mu_X$ denote the unique $G$-invariant Radon probability measure on $X.$ 

\smallskip   

We may identify $X$ and $\{g\mathcal{P} : g \in G\}$ via the bijection $g\Gamma \longleftrightarrow g\mathcal{P}$; we then equip $\{g\mathcal{P} : g \in G\}$ with the quotient topology of $X$ by declaring the aforementioned bijection to be a homeomorphism. Given any function $f : \R^n \to \R_{\geq 0},$ we define its ${{\mathcal{P}}}$-\textsl{Siegel transform} $\ds \widehat{f}^{{\,}^{\mathcal{P}}} : X \to {[0, \infty]}$ by \[ \widehat{f}^{{\,}^{\mathcal{P}}}(g\Gamma) = \widehat{f}^{{\,}^{\mathcal{P}}}(g\mathcal{P}) := \sum_{\vv \in \mathcal{P}} f(g\vv). \] We note that if $f$ is Borel measurable, then $\ds \widehat{f}^{{\,}^{\mathcal{P}}}$ is $\mu_X$-measurable. The Siegel and Rogers type axioms that are recalled below were introduced by the authors in \cite{KS}; they were named after the groundbreaking results of Siegel \cite{Siegel} and Rogers \cite{RogSet, MeanRog}. 


\begin{defn}[{\cite[Definition 2.1]{KS}}]\label{SiegelRogers} Let $G$ and ${{\mathcal{P}}}$ be as above.
\smallskip

\item[ \rm (i)] We say that $G$ is of ${{\mathcal{P}}}$-\textsl{Siegel type} if there exists a constant $c = {c}_{{{\mathcal{P}}}} \in \R_{>0}$ such that for any bounded and compactly supported Borel measurable function $f: \R^n \to \R_{\geq 0}$ we have \[ \int_{X}\widehat{f}^{{\,}^{{\mathcal{P}}}} \, \dint\mu_X = c \int_{\R^n}f \, \dint m.\] 

\smallskip

\item[ \rm (ii)]  
Let $\ds r \in \R_{\geq 1}$ be given. We say that $G$ is of $\ds \left({{\mathcal{P}}}, r \right)$-\textsl{Rogers type} if there exists a constant $D = D_{\mathcal{P}, r} \in \R_{>0}$ such that for any bounded Borel $E \subset \R^n$ with $m(E) > 0$ we have \[\norm{\widehat{\mathbbm{1}_E}^{{}_{{\mathcal{P}}}} - \left(\int_{X}\widehat{\mathbbm{1}_E}^{{}_{{\mathcal{P}}}} \,d\mu_X \right) \mathbbm{1}_X}_{r} \leq D \cdot 
m(E)^{1/r}. \]  \end{defn}

\smallskip

There are various interesting examples of $(G, \mathcal{P})$ that satisfy both conditions (i) (Siegel) and (ii) (Rogers) of Definition \ref{SiegelRogers}. For the convenience of the reader, we now record some examples that were already discussed in \cite{KS}:~see \cite[Theorems 2.5, 2.6, and 2.8]{KS} and the references therein for details. Below, $\zeta$ denotes the Euler\textendash Riemann zeta function.

\begin{thm}[{\cite[Theorems 2.5, 2.6, and 2.8]{KS}}]\label{examplesSiegelRogers}
\begin{itemize}
\item[]
\item[\rm (i)] The group $\operatorname{ASL}_n(\R)$ is of $\Z^n$-Siegel type with $\ds {c}_{\Z^n} = 1$ and of $\left(\Z^n, 2\right)$-Rogers type.  
\item[\rm (ii)] Suppose $n \geq 3.$ Then the group $\operatorname{SL}_n(\R)$ is of $\Z_{\mathrm{pr}}^n$-Siegel type with $\ds {c}_{ \Z_{\mathrm{pr}}^n} = 1/\zeta(n)$,  of ${\Z^n_{\ne0}}$-Siegel type with $\ds {c}_{{\Z^n_{\ne0}}} = 1,$ of $\ds \left(\Z_{\mathrm{pr}}^n, 2\right)$-Rogers type, and of $(\Z^n_{\ne0},2)$-Rogers type.
\item[(iii)] The group $\operatorname{SL}_2(\R)$ is of $\Z_{\mathrm{pr}}^2$-Siegel type with $\ds {c}_{ \Z_{\mathrm{pr}}^2} = 1/\zeta(2)$, of ${\Z^2_{\ne0}}$-Siegel type with $\ds {c}_{{\Z^2_{\ne0}}} = 1$, and of $\ds \left(\Z_{\mathrm{pr}}^2, 2\right)$-Rogers type. For every $p \in [1, 2) \subset \R,$ the group $\operatorname{SL}_2(\R)$ is of $(\Z^2_{\ne0},p)$-Rogers type.
\item[\rm (iv)] Suppose $\ds n \in 2\left(\Z_{\geq 2}\right).$ Then the group $\operatorname{Sp}_n(\R)$ is of $\Z_{\mathrm{pr}}^n$-Siegel type with $\ds {c}_{ \Z_{\mathrm{pr}}^n} = 1/\zeta(n)$, of $\Z^n_{\ne0}$-Siegel type with $\ds {c}_{\Z^n_{\ne0}} = 1,$ of $\ds \left(\Z_{\mathrm{pr}}^n, 2\right)$-Rogers type, and of $(\Z^n_{\ne0},2)$-Rogers type. \end{itemize}
\end{thm}

\noindent  Using the recent work of Ghosh\textendash Kelmer\textendash Yu \cite{GKY_a}, one can improve upon Theorem \ref{examplesSiegelRogers}(ii) in an arithmetically interesting fashion. For any $q \in \N,$ let $\Z^n_{\rm pr}(q)$ denote the set of all $\vv \in \Z^n_{\rm pr}$ that are congruent modulo $q$ to the first standard basis vector $\mathbf{e}_1 := (1, 0, \dots , 0) \in \Z^n \subset \R^n$:~more precisely, we have $\vv \in \Z^n_{\rm pr}(q)$ if and only if $\vv \in \Z^n_{\rm pr}$ and $\vv - \mathbf{e}_1 \in q\Z^n.$ For any $q \in \N,$ let $\zeta_q$ denote the modified Euler\textendash Riemann zeta function given by $\ds \zeta_q(s) := \sum_{k \in \N  :  \gcd(k, q)=1} k^{-s}.$ (Note that $\Z^n_{\rm pr}(1) = \Z^n_{\rm pr}$ and $\zeta_1 = \zeta.$) The following then holds.

\begin{thm}[{\cite[Theorem 1.4]{GKY_a}}]\label{congruenceprimitive}
Suppose $n \geq 3.$ Let $q \in \N$ be arbitrary.  Then the group $\operatorname{SL}_n(\R)$ is of $\Z_{\mathrm{pr}}^n(q)$-Siegel type with $\ds {c}_{ \Z_{\mathrm{pr}}^n(q)} = \frac1{q^n \, \zeta_q(n)}$ and of $\ds \left(\Z_{\mathrm{pr}}^n(q), 2\right)$-Rogers type. 
\end{thm}
\begin{proof} Set $\mathcal{P} := \Z_{\mathrm{pr}}^n(q).$ Let $f: \R^n \to \R_{\geq 0}$ be a bounded and compactly supported Borel measurable function. It follows from \cite[Theorem 1.4, (1.9)]{GKY_a} and \cite[Lemma 2.2]{GKY_a} that  
\begin{equation}\label{meanofcongruence}
\int_{X}\widehat{f}^{{\,}^{{\mathcal{P}}}} \, \dint\mu_X = \frac{\zeta(n)}{q^n \, \zeta_q(n)} \cdot \frac{1}{\zeta(n)} \int_{\R^n}f \, \dint m = \frac{1}{q^n \, \zeta_q(n)} \int_{\R^n}f \, \dint m. 
\end{equation} This proves the first assertion. 

Using \cite[Theorem 1.4, (1.10)]{GKY_a}, \cite[Lemma 2.2]{GKY_a}, and the Cauchy\textendash Schwarz inequality, we obtain \begin{equation}\label{secondcongruence}
\frac{q^n \, \zeta_q(n)}{\zeta(n)} \int_{X} \left(\widehat{f}^{{\,}^{{\mathcal{P}}}} \right)^2 \, \dint\mu_X \leq \frac{1}{\zeta(n) \, q^n \, \zeta_q(n)} \left(\int_{\R^n}f \, \dint m\right)^2 + \frac{2}{\zeta(n)} \, \norm{f}_2^2. 
\end{equation} We then infer from \eqref{meanofcongruence} and \eqref{secondcongruence} that 
\begin{equation*}
\int_{X} \left(\widehat{f}^{{\,}^{{\mathcal{P}}}} \right)^2 \, \dint\mu_X - \left(\int_{X}\widehat{f}^{{\,}^{{\mathcal{P}}}} \, \dint\mu_X \right)^2 \leq \frac{2}{q^n \, \zeta_q(n)} \, \|f\|_2^2. \end{equation*} The second assertion now follows easily.  \end{proof}

\begin{rmk} \rm  
The same proof applies to show that for any $q \in \N,$ the group $\operatorname{SL}_2(\R)$ is of $\Z_{\mathrm{pr}}^2(q)$-Siegel type with $\ds {c}_{ \Z_{\mathrm{pr}}^2(q)} = \frac1{q^2 \, \zeta_q(2)}$:~see \cite[Remark 1.11]{GKY_a}.            
\end{rmk}

The Siegel and Rogers axioms are expedient because they can be used in tandem with the Borel\textendash Cantelli Lemma (for the purpose of uniform approximation) and to prove analogues of W.\,M.~Schmidt's famous counting result \cite[Theorems 1 and 2]{Schmidt} (for the purpose of asymptotic approximation). We now recall a result from \cite{KS} that records the relevant consequences of the Siegel and Rogers axioms.  

\begin{thm}[{\cite[Theorem 2.9]{KS}}]\label{genericcounting} 
Let $G$ be a closed subgroup of \ASL, and let ${{\mathcal{P}}}$ be a subset of $\Z^n.$ Suppose $G$ is of $\mathcal{P}$-Siegel type with $c=c_{\mathcal{P}}$. Let $E$ be a Borel measurable subset of $\R^n.$ 
\smallskip  
\begin{itemize}
\item[\rm (i)] If $m(E) < \infty,$ then $\ds \mu_X\left(\left\lbrace \Lambda \in X: \#\left(\Lambda \cap E\right) < \infty \right\rbrace\right) = 1.$  
\end{itemize}
\smallskip
For the remaining statements of this theorem, suppose in addition to the preceding hypotheses that we are given $r \in \R_{> 1}$ for which $G$ is of $\ds \left(\mathcal{P}, r \right)$-Rogers type. 
\smallskip
\begin{itemize}
\item[\rm (ii)] Suppose $m(E) = \infty$. Let $\|\cdot\|$ be a norm on $\R^n$. Then for $\mu_X$-almost every $\Lambda \in X,$ we have  \eq{limit}{\lim_{t \to \infty } \frac{\# \left\lbrace\xx\in\left(\Lambda \cap E\right) : \|\xx\| \leq t \right\rbrace}{c\, m\left(\{\xx\in  E: \|\xx\| \leq t \}\right)} = 1.} 
\smallskip
In particular,  $\ds \mu_X\left(\left\lbrace \Lambda \in X: \#\left(\Lambda \cap E\right) = \infty \right\rbrace\right) = 1.$   
\smallskip  
\item[\rm (iii)] Let $\ds\{F_k\}_{k \in \N}$ be a collection of 
Borel measurable subsets of $\R^n$ with $0 < m(F_k) < \infty$ for each $k \in \N$. Suppose $\ds \sum_{k=1}^{\infty} m(F_k)^{1-r} < \infty.$ Then the following holds:~for $\mu_X$-almost every $\Lambda \in X$ there exists  $k_\Lambda \in \N$ such that $\ds \Lambda \cap F_k \neq \varnothing$ whenever $k \geq k_\Lambda$.  
\end{itemize}           
\end{thm}  

\begin{rmk}[{\cite[Remark 2.12]{KS}}] \rm Let $G$ be a closed subgroup of $\operatorname{ASL}_n(\R)$, and let $\mathcal{P}$ be a subset of $\Z^n.$ If $G$ is of $\mathcal{P}$-Siegel type, then the triangle inequality in $L^1(X)$ implies that $G$ is of $(\mathcal{P}, 1)$-Rogers type. Notice, however, that the hypotheses of Theorem \ref{genericcounting}(iii) can never be satisfied when $r \in \R_{>1}$ is replaced by $1.$ We now give an example to show that Theorem \ref{genericcounting}(ii) need not be true when $r \in \R_{>1}$ is replaced by $1.$ Let $G := \R^n.$ Then $G$ is a closed subgroup of $\operatorname{ASL}_n(\R)$ that is of $\Z^n$-Siegel type with $c_{\Z^n} = 1$; thus, $G$ is of $(\Z^n, 1)$-Rogers type. Let $\varepsilon \in (0, 1) \subset \R$ be given, and define $\ds U_\varepsilon := \R^{n-1} \times \left( \frac{1-\varepsilon}{2}, \frac{1+\varepsilon}{2}\right) \subseteq \R^{n-1} \times (0, 1) \subset \R^n.$ Since $n \geq 2$, we have $\ds m\left(U_\varepsilon\right) = \infty.$ Now note that $\ds \mu_X\left(\left\lbrace \Lambda \in X: \#\left(\Lambda \cap U_\varepsilon\right) = \infty \right\rbrace\right) = 1-\varepsilon < 1.$  \end{rmk}

\smallskip

{In the authors' previous paper \cite{KS}, the assumption on the function $f: \R^n \to \R^\ell$ whose values near zero were being approximated was \textit{subhomogeneity}; we recall here the definition of this property. }  
\begin{defn}\label{subhom} Let $f: \R^n \to \R^\ell.$ We say that $f$ is \textsl{subhomogeneous} if it is Borel measurable and there exists ${\delta} = {\delta}(f) \in \R_{>0}$ such that for each $t \in (0, 1) \subset \R$ and each $\xx \in \R^n$ we have $\ds |f(t\xx)| \les t^{{\delta}} |f(\xx)|.$ \end{defn}  
\noindent In this paper, we wish to prove analogues of \cite[Theorems 3.4 and 3.8]{KS} in which the subhomogeneous function $f : \R^n \to \R^\ell$ therein is replaced by a function of the form $_{\pmb \xi} F,$ where $F : \R^n \to \R^\ell$ is some sufficiently well-behaved Borel measurable function and ${\pmb \xi} \in \R^\ell$. The technical starting point for doing so is noting that the subhomogeneity assumption on $f$ in \cite{KS} was needed only in order to invoke the conclusions of two important lemmata from \cite{KS}; we now recall the statements of these lemmata for the convenience of the reader. 
\begin{lem}[{\cite[Lemmata 3.1 and 3.5]{KS}}]\label{3.1and3.5}
Let $\ds f 
: \R^n \to \R^{\ell}$ be subhomogeneous, and let $\ds {\delta} = {\delta}_f \in \R_{>0}$ as in Definition \ref{subhom}. Let $\eta$ and $\nu$ be arbitrary norms on $\R^n$. The following then hold. 
\begin{itemize} 
\item[\rm (i)] Let $s \in \R_{>0}$. Then $m(A_{f, s\psi, \eta}) < \infty$ if and only if $m(A_{f, \psi, \nu}) < \infty$.
\item[\rm (ii)] Let  $t \in (0, 1),$ $T \in \R_{>0},$ and ${\pmb\varepsilon} \in \left(\R_{>0}\right)^\ell.$ Then $\ds {t} B_{f, {\pmb\varepsilon}, \nu, T} \subseteq B_{f, {t}^{{\delta}}{\pmb\varepsilon}, \nu, {t} T}$.   
\item[\rm (iii)] {There exists $\ds C^* = C_{\nu, \eta}^* \in \R_{\geq 1}$ such that for each $\ds C \in \R_{\geq {C^*}},$ each $T \in \R_{>0},$ and each ${\pmb\varepsilon} \in \left(\R_{>0}\right)^\ell,$ we have $\ds B_{f, {\pmb\varepsilon}, \nu, T} \subseteq C \, B_{f, C^{-{\delta}}{\pmb\varepsilon}, \eta, T}$ and $\ds B_{f, {\pmb\varepsilon}, \eta, T} \subseteq C \, B_{f, C^{-{\delta}}{\pmb\varepsilon}, \nu, T} $. } 
\end{itemize}
\end{lem}

\noindent In essence, the above lemma shows that the subhomogeneity of $f$ is sufficient to ensure that the volumes of the sets $A_{f, \psi, \nu}$ and $B_{f, {\pmb\varepsilon}, \nu, T}$ are well-behaved under change-of-norms and under scaling of any of the arguments $\psi, {\pmb\varepsilon}, T$ by arbitrary elements of $\R_{>0}.$  Let us also recall another lemma from \cite{KS} to which we shall need to refer.

\begin{lem}[{\cite[Lemma 3.2]{KS}}] \label{KSLemma3.2}
Let $\ds \psi = \left(\psi_1, \dots , \psi_{\ell}\right): \R_{\geq 0} \to \left(\R_{>0}\right)^{\ell}$ be regular and nonincreasing. Then the following holds: for any $c \in \R_{\geq 0}$ there exists $s \in \R_{>0}$ such that for each $x \in [0, c]$ and each $y \in \R_{>c},$ one has $\ds \psi(y-x) \les s\psi(y).$
\end{lem}

We now  formulate definitions that axiomatize   {(i) of Lemma \ref{3.1and3.5} and certain desirable properties furnished by (ii) and (iii) of the same lemma.  }

\ignore{ \rmv{~see Definitions \ref{asympmeasureinvariancedefn} and \ref{unifmeasureinvariancedefn}. The analogues of \cite[Theorems 3.4 and 3.8]{KS} then follow immediately:~see Theorems \ref{asymptoticmain} and \ref{mainuniformresult}.  We then prove a volume estimate in Theorem \ref{submersionvol} that allows us to establish a generalization of Theorem \ref{simplesubmerapprox}:~see Theorem \ref{submerapprox}. Lastly, we consider two examples that, due to the presence of singularities, are not within the purview of Theorems \ref{submersionvol} and \ref{submerapprox}:~see Theorems \ref{productcoord} and \ref{khintchine}. }            \comm{Let us remove this description for now, we'll put it back later.}}


\begin{defn}\label{asympmeasureinvariancedefn} Let $f : \R^n \to \R^{\ell}$ be Borel measurable, and let $\psi : \R_{\geq 0} \to \left(\R_{>0}\right)^\ell$ be regular and nonincreasing. 
We say that the pair $(f, \psi)$ is \textsl{asymptotically acceptable} if the following holds:~for any $s \in \R_{>0}$ and any norms $\nu_1, \nu_2$ on $\R^n,$ we have \eq{asympproperty}{ m\left( A_{f, \psi, \nu_1} \right) < \infty \quad \text{if and only if} \quad m\left( A_{f, s\psi, \nu_2} \right) < \infty. } 

\end{defn}


\begin{defn}\label{unifmeasureinvariancedefn}
Let $f : \R^n \to \R^{\ell}$ and $\psi : \R_{\geq 0} \to \left(\R_{>0}\right)^\ell$ each be Borel measurable. We say that the pair $(f, \psi)$ is \textsl{uniformly acceptable} if the following holds:~for any $s_1, s_2, s_3, s_4 \in \R_{>0}$ and any norms $\nu_1, \nu_2$ on $\R^n,$  we have 
\eq{unifproperty}{ 0 < \liminf_{T \to \infty} \frac{m\left(B_{f, s_3\psi(T), \nu_2, s_4 T}\right)}{m\left(B_{f, s_1\psi(T), \nu_1, s_2 T}\right)} \leq \limsup_{T \to \infty} \frac{m\left(B_{f, s_3\psi(T), \nu_2, s_4 T}\right)}{m\left(B_{f, s_1\psi(T), \nu_1, s_2 T}\right)} < \infty.}   \end{defn}

\noindent  Informally speaking, the pair $(f, \psi)$ is asymptotically (respectively, uniformly) acceptable if the measure of sets of the form $A_{f, \psi, \nu}$  (respectively, $B_{f, \psi(T), \nu, T}$) does not change drastically when one multiplies $\psi$ (respectively, $\psi$ and the occurrence of $T$ that is not the argument of $\psi$) by some arbitrary constant (respectively, constants) and changes the norm $\nu$.

\begin{rmk}
Note that if $f: \R^n \to \R^\ell$ is subhomogeneous and $\psi : \R_{\geq 0} \to \left(\R_{>0}\right)^{\ell}$ is regular and nonincreasing, then Lemma \ref{3.1and3.5} implies that the pair $(f, \psi)$ is both asymptotically and uniformly acceptable. Therefore, the  framework of \cite{KS} can be subsumed into that of the present paper. \end{rmk}

\ignore{\begin{rmk} \rm 
Definition \ref{unifmeasureinvariancedefn} may appear rather convoluted, but it serves a simple technical purpose that will become obvious in the proof of Theorem \ref{mainuniformresult}(i):~the properties that this definition captures enable us to use the limit comparison test.  
\end{rmk}

\begin{rmk} \rm   
Note that if $f: \R^n \to \R^\ell$ is subhomogeneous and $\psi : \R_{\geq 0} \to \left(\R_{>0}\right)^{\ell}$ is regular and nonincreasing, then Lemma \ref{3.1and3.5} implies that the pair $(f, \psi)$ is both asymptotically acceptable and uniformly acceptable.  \end{rmk}} Let us now state and prove a lemma that will simplify various proofs in \S \ref{acc} that concern the verification of asymptotic and uniform acceptability. 

\begin{lem}\label{scalingonly} 
Let $f : \R^n \to \R^{\ell}$ and $\psi : \R_{\geq 0} \to \left(\R_{>0}\right)^\ell$ each be Borel measurable. 
\begin{itemize}
\item[{\rm (i)}] Suppose that $\psi$ is regular and nonincreasing. Suppose that there exists a norm $\nu$ on $\R^n$ such that for any $s \in \R_{>0}$ we have \[ m\left( A_{f, \psi, \nu} \right) < \infty \quad \text{if and only if} \quad m\left( A_{f, s\psi, \nu} \right) < \infty. \] Then the pair $(f,\psi)$ is asymptotically acceptable.    

\vspace{0.05in}

\item[{\rm (ii)}] Suppose that there exists a norm $\nu$ on $\R^n$ such that for any $s_1, s_2, s_3, s_4 \in \R_{>0}$ we have \[ 0 < \liminf_{T \to \infty} \frac{m\left(B_{f, s_3\psi(T), \nu, s_4 T}\right)}{m\left(B_{f, s_1\psi(T), \nu, s_2 T}\right)} \leq \limsup_{T \to \infty} \frac{m\left(B_{f, s_3\psi(T), \nu, s_4 T}\right)}{m\left(B_{f, s_1\psi(T), \nu, s_2 T}\right)} < \infty. \] Then the pair $(f,\psi)$ is uniformly acceptable.   
\end{itemize}
\end{lem}   
\begin{proof}
Let $\eta$ be an arbitrary norm on $\R^n.$ 
\begin{itemize}
\item[(i)] Let $a = a(\psi) \in \R_{>1}$ and $b = b(\psi) \in \R_{>0}$ be as in Definition \ref{basicapdefns}. Let $\nu$ be a norm on $\R^n$ as in the hypotheses of (i). Fix $k \in \N$ for which $\ds a^{-k} \, \nu \leq \eta \leq a^k \, \nu.$ For every $\xx \in \R^n,$ we have \[ \psi\big(\eta(\xx)\big) \les \psi\big(a^{-k} \, \nu(\xx) \big) \les b^{-k} \, \psi\big(\nu(\xx)\big). \] For every $\xx \in \R^n,$ we similarly have \[ b^k \, \psi\big(\nu(\xx)\big) \les \psi\big(\eta(\xx)\big). \] It follows that for each $s \in \R_{>0},$ we have \[ m\left( A_{f, {b^k s \psi}, \nu} \right) \leq m\left( A_{f, {s\psi}, \eta} \right) \leq m\left( A_{f, {b^{-k} s \psi}, \nu} \right). \] 
{This implies the desired result.}

\vspace{0.05in}

\item[(ii)] We do not assume that $\psi$ is regular, and we do not assume that $\psi$ is nonincreasing. Let $\nu$ be a norm on $\R^n$ as in the hypotheses of (ii). Fix $C \in \R_{>1}$ for which $\ds C^{-1} \, \nu \leq \eta \leq C \, \nu.$ Let $s_1, s_2, T \in \R_{>0}$ be arbitrary. We clearly have \[ m\left(B_{f, s_1\psi(T), \nu, {C^{-1} s_2 T}}\right) \leq m\left(B_{f, s_1\psi(T), \eta, s_2 T}\right) \leq m\left(B_{f, s_1\psi(T), \nu, {C s_2 T}}\right). \] The desired result now follows.  \end{itemize} \end{proof}

We now proceed to establish our main theorems.  

\smallskip

\begin{thm}\label{asymptoticmain} Let $G$ be a closed subgroup of $\operatorname{ASL}_n(\R)$, and let ${{\mathcal{P}}}$ be a subset of $\Z^n.$ Suppose $G$ is of $\mathcal{P}$-Siegel type. Let $\ds {c} = {c}_{{{\mathcal{P}}}}$ be as in Definition {\rm\ref{SiegelRogers}(i)}. Let $f : \R^n \to \R^{\ell}$ be Borel measurable. Let $\psi : \R_{\geq 0} \to \left(\R_{>0}\right)^\ell$ be regular and nonincreasing. 
Suppose that the pair $(f, \psi)$ is asymptotically acceptable. Let $\eta$ and $\nu$ be arbitrary norms on $\R^n.$ 
\vspace{0.055in}  
\begin{itemize}
\item[\rm (i)] Suppose $\ds m\left( A_{f, \psi, \eta} \right) < \infty.$ Then for almost every $g \in G$ the function $f\circ g$ is not $(\psi, \nu, \mathcal{P})$-approximable.   
\smallskip

\item[\rm (ii)] Suppose $\ds m\left( A_{f, \psi, \eta} \right) = \infty,$  and suppose that we are given $r \in \R_{> 1}$  for which $G$ is of $\ds \left(\mathcal{P}, r\right)$-Rogers type.  Then \ignore{for each nonempty compact subset $K$ of $G$  there exist constants $D_K \in \R_{\geq 1},$ $E_K \in \R_{\geq 0}$  and $\ds J_K \in \R_{\geq 1}$ such that for $\mu_G$-almost every $g \in K$ we have \begin{align}\label{limsup}
\limsup_{T \to \infty}\frac{{\#}\left\lbrace \vv \in {{\mathcal{P}}} : \left|( f \circ  g )(\vv)\right| \les \psi\big( \nu(\vv) \big) \ \operatorname{and} \ 2 D_K E_K < \nu(\vv) \leq T \right\rbrace}{m\left(\left\lbrace \ttt \in \R^n : \left| f(\ttt) \right| \les J_K \psi\big(\nu(\ttt)\big) \ \operatorname{and} \ {E_K} < \nu(\ttt) \leq D_K T + E_K \right\rbrace\right)} \leq {{c}_{{\mathcal{P}}}} \qquad \end{align} and 
\begin{align}\label{liminf} \ \quad 
\liminf_{T \to \infty} \frac{{\#}\left\lbrace \vv \in {{\mathcal{P}}} : \left|( f \circ  g )(\vv)\right| \les \psi\big( \nu(\vv)\big) \ \operatorname{and} \ {E_K D_K^{-1}} < \nu(\vv) \leq D_K E_K + D_K T \right\rbrace}{m\left(\left\lbrace \ttt \in \R^n : \left| f(\ttt) \right|  \les J_K^{-1}\psi\big((\nu(\ttt)\big) \ \operatorname{and} \ 2E_K < \nu(\ttt) \leq T \right\rbrace\right)} \geq {{c}_{{\mathcal{P}}}}. \end{align}

\smallskip

\noindent   If $K \subseteq \SL,$ then each of the inequalities in \eqref{limsup} and \eqref{liminf} holds with $E_K = 0.$

\noindent   In particular,}for almost every $g \in G$ the function $f \circ g$ is $(\psi, \nu, {{\mathcal{P}}})$-approximable. \end{itemize} \end{thm}

\begin{proof} 
Let $a = a(\psi) \in \R_{>1}$ and $b = b(\psi) \in \R_{>0}$ be as in Definition \ref{basicapdefns}. Let us denote elements of ${\rm ASL}_n(\R)$ by $\langle h,\zz \rangle$, where $h \in \SL$ and $\zz \in \R^n$; that is, 
$\langle h,\zz\rangle : \R^n \to \R^n$ is the affine map given by $\xx \mapsto \zz + h\xx$. For any $h \in \SL$, let $\|h\|$ denote the operator norm of $h$ that is given by \eq{norm}{\ds \|h\|:= \sup\left\lbrace \nu(h\xx) : \xx \in \R^n \ \text{and} \ 
\nu(\xx) \leq 1 \right\rbrace.}           

\smallskip

We first prove (i). Suppose $\ds m\left(A_{f, \psi, \eta}\right) < \infty.$ In view of \equ{asympproperty}, it follows that for every $M \in \N$ we have $\ds m\left(A_{f, M\psi, \nu}\right) < \infty.$ Theorem \ref{genericcounting}(i) then implies that for every $M \in \N$ we have \[\mu_{X}\left(\left\{\Lambda \in X: {\#}\left(\Lambda \cap A_{f, M\psi, \nu}\right) = \infty \right\}\right) = 0. \] Hence, the set \[ S_1 :=   \bigcup_{M \in \N}\left\lbrace g \in G: {\#}\left(g\mathcal{P} \cap A_{f, M\psi, \nu}\right) = \infty \right\rbrace \] satisfies $\mu_G({S_1}) = 0$. Now let $g = \langle h, \zz \rangle$ be any element of $G$ for which \eq{appr}{f \circ g\text{ is $\left(\psi, \nu, {{\mathcal{P}}}\right)$-approximable}.} Let $\ds D := \max\left\lbrace \|h\|, \norm{h^{-1}}\right\rbrace,$ and let $\ds E := \nu(\zz).$ Let $k\in\N$ be such that $\ds a^{k} > D.$ 
Since $\psi$ is regular and nonincreasing, it follows from {Lemma \ref{KSLemma3.2}}
that there exists $F \in \R_{>0}$ for which the following is true:~for each $x \in [0, E]$ and each $\ds y \in (E, \infty),$ we have $\psi(y-x) \les F \psi(y).$ Let $N$ be any element of $\N$ with $N > b^{-k}F.$ {A simple argument then yields the inclusion} \eq{nullinclusion}{  {g\left(\left\lbrace \xx \in 
A_{f \circ g, \psi, \nu}  : \nu(\xx) > 2 DE \right\rbrace\right) \subseteq 
A_{f, N\psi, \nu}.}} In light of \equ{appr} and \equ{nullinclusion}, it follows that $\ds {\#}\left(g\mathcal{P} \cap A_{f, N\psi, \nu}\right) = \infty.$ We conclude that $g$ belongs to the $\mu_G$-null set ${S_1}$. 
This completes the proof of (i).

\smallskip

{The proof of (ii) proceeds along similar lines}. Suppose $\ds m\left( A_{f, \psi, \eta} \right) = \infty,$  and suppose that we are given $r \in \R_{> 1}$  for which $G$ is of $\ds \left(\mathcal{P}, r\right)$-Rogers type. It then follows from \equ{asympproperty} and Theorem \ref{genericcounting}(ii) that the set  \[ S_2 :=   \bigcup_{M \in \N}\left\lbrace g \in G: {\#}\left(g\mathcal{P} \cap A_{f, \psi/M, \nu}\right) < \infty \right\rbrace \] is $\mu_G$-null.  Now let $g$ be any element of $G$ for which $f \circ g$ is not $\left(\psi, \nu, {{\mathcal{P}}}\right)$-approximable. Arguing as in the proof of (i), we conclude that 
{$g\in S_2$, which finishes}
the proof of (ii).  \end{proof}

\begin{rmk}\label{quant} \rm Using \equ{limit} and arguing as in the proof of \cite[Theorem 3.4]{KS}, it is possible to enhance the qualitative conclusion of Theorem \ref{asymptoticmain}(ii) in a quantitative fashion. Since we are primarily interested in qualitative results, we decided to forego quantitative arguments. 
\end{rmk}

To state the next theorem, we need the following definition.       

\begin{defn}\label{subsequences}
\begin{itemize} \item[]
\item Let $\ds t_\bullet = \left(t_k\right)_{k \in \N}$ be any strictly increasing sequence of elements of $\R_{>0}$ with $\ds \lim_{k \to \infty} t_k = \infty.$ We then say that $f$ is $\ds t_\bullet$-\textsl{uniformly} $\left( \psi, \nu, \mathcal{P} \right)$-approximable if {$B_{f, \psi(t_k), \nu, t_k} \cap \mathcal{P}  \neq \varnothing$ for each sufficiently large $k \in \N$}.   
\item Let $\ds t_\bullet = \left(t_k\right)_{k \in \N}$ be any strictly increasing sequence of elements of $\R_{>0}$ with $\ds \lim_{k \to \infty} t_k = \infty.$ We say that $t_\bullet$ is \textsl{quasi-geometric} if, in addition to the preceding, the set $\ds \left\lbrace t_{k+1}/t_k : k \in \N \right\rbrace$ is bounded from above.    
\end{itemize}   
\end{defn}


\begin{thm}\label{mainuniformresult} Let $G$ be a closed subgroup of $\operatorname{ASL}_n(\R)$, and let ${{\mathcal{P}}}$ be a subset of $\Z^n.$ Suppose $G$ is of $\mathcal{P}$-Siegel type, and suppose that we are given $r \in \R_{>1}$ for which $G$ is of $\left(\mathcal{P}, r\right)$-Rogers type. Let $f : \R^n \to \R^{\ell}$ and $\psi : \R_{\geq 0} \to \left(\R_{>0}\right)^\ell$ each be Borel measurable. Suppose that the pair $(f, \psi)$ is uniformly acceptable. Let $u_\bullet = (u_k)_{k \in \N}$ be any strictly increasing sequence of elements of $\R_{>1}$ with $\ds \lim_{k \to \infty} u_k = \infty.$ Suppose that there exists some norm $\eta$ on $\R^n$ for which \eq{conv}{\inf_{N \in \N}\sum_{k=N}^{\infty}  m\left(B_{f, \psi\left(u_k\right), \eta, u_k}\right)^{1-r} < \infty.} Let $\nu$ be an arbitrary norm on $\R^n$. The following then hold. 
\vspace{0.055in}
\begin{itemize}
\item[\rm (i)] For almost every $g \in G$ the function $f \circ g$ is $u_\bullet$-uniformly $\ds \left(\psi, \nu, \mathcal{P}\right)$-approximable. 
\vspace{0.055in}
\item[\rm (ii)] Suppose further that $\psi$ is regular and nonincreasing and that the sequence $u_\bullet$ is quasi-geometric. Then for almost every $g \in G$ the function $f \circ g$ is uniformly $(\psi, \nu, {{\mathcal{P}}})$-approximable. \end{itemize} \end{thm}

\begin{proof} 
We argue as in the proof of \cite[Theorem 3.8]{KS}, appealing to the uniform acceptability of the pair $(f,\psi)$ (instead of appealing to Lemma \ref{3.1and3.5}).     

\vspace{0.050in}   
\begin{itemize}
\item[\rm (i)] As in the proof of Theorem \ref{asymptoticmain}, we denote elements of $\operatorname{ASL}_n(\R)$ by $\langle h,\zz \rangle.$ For any $h \in \SL$, let $\|h\|$ denote the operator norm of $h$ 
given by {\equ{norm}}. 
Define $\pi : {\rm ASL}_n(\R) \to \SL$ and $\rho : {\rm ASL}_n(\R) \to \R^n$ by $\ds \pi : \langle h, \zz \rangle \mapsto h$ and $\ds \rho : \langle h, \zz \rangle \mapsto \zz.$ Notice that each of these maps is continuous. Let $K$ be a  nonempty compact subset of $G$ with $K = K^{-1}$, and define \[ { D_K := \sup \left\lbrace \|h\| : h \in \pi(K)  \right\rbrace < \infty \quad \quad \quad \text{and} \quad \quad \quad E_K := \sup \left\lbrace \nu(\zz) : \zz \in \rho(K) \right\rbrace < \infty.} \] Fix $L \in \N$ such that for each $k \in \Z_{\geq L}$ we have $u_k > 2 E_K.$ 
It follows from \equ{unifproperty} and \equ{conv} that $\ds \inf_{N \in \N}\sum_{k=N}^{\infty}  m\left(B_{f, \psi\left(u_k\right), \nu , {u_k}/{2D_K} }\right)^{1-r} < \infty.$ We then apply Theorem \ref{genericcounting}(iii) to obtain the following:~for almost every $g \in G$ there exists $M_{g} \in \Z_{\ge L}$ such that for each $k \in \Z$ with $k \geq M_g$ there exists some $\vv_{g, k} \in \mathcal{P}$ with \begin{equation}\label{somelabel} \nu\left(g \vv_{g, k}\right) \leq \frac{u_k}{2D_K} \quad \quad  \quad \text{and} \quad \quad \quad  \left|f\left(g\vv_{g, k}\right)\right| \les \psi\left(u_k\right).   
\end{equation} {For any such $\mu_G$-generic $g$ that belongs to $K$ and any integer $k \geq M_g,$ we have $\nu\left(\vv_{g, k}\right) \leq u_k$; this may be proved by appealing to the first inequality in \eqref{somelabel} and arguing as in the proof of \cite[Theorem 3.8(i)]{KS}. We thus conclude that for $\mu_G$-almost every $g \in K$ the function $f \circ g$ is $u_\bullet$-uniformly $(\psi, \nu, \mathcal{P})$-approximable. Since $G$ is $\sigma$-compact, the desired result follows. }

\smallskip

\item[\rm (ii)] Let $a = a(\psi) \in \R_{>1}$ and $b = b(\psi) \in \R_{>0}$ be as in Definition \ref{basicapdefns}. Fix $j \in \N$ for which $$ \sup\left\lbrace {u_{k+1} /u_k} : k \in \N \right\rbrace < a^j.$$ (This is possible because $u_\bullet$ is quasi-geometric.) Appealing once again to \equ{unifproperty} and \equ{conv}, we infer $\ds \inf_{N \in \N} \sum_{k=N}^{\infty}  (m\left(B_{f, b^j\psi\left(u_k\right), \eta, u_k}\right) ^{1-r} < \infty.$ Statement (i) of this theorem implies that for almost every $g \in G$ the function $f \circ g$ is $u_\bullet$-uniformly $\left(b^j\psi, \nu, \mathcal{P}\right)$-approximable. Now let $\ds h : \R^n \to \R^\ell$ be any function that is $u_\bullet$-uniformly $\left(b^j\psi, \nu, \mathcal{P}\right)$-approximable. Fix ${M} \in \N$ such that for each $k \in \Z_{\geq {M} }$ the set $\ds B_{h, b^j\psi(u_k), \nu, u_k} \cap \mathcal{P}$ is nonempty. Let $\ds T \in \left( u_{{M} +2}, + \infty\right)$ be arbitrary. Then there exists $i \in \Z_{\geq {M} +2}$ for which $\ds u_i \leq T \leq u_{i+1}.$ Note that there exists $\vv \in \mathcal{P}$ with $\nu(\vv) \leq u_i$ and $\ds |h(\vv)| \les b^j \psi(u_i).$ We then have $\nu(\vv) \leq u_i \leq T$ and  \begin{equation*} |h(\vv)| \les b^j \psi(u_i) \les b^j b^{-j} \psi\left(a^j u_i \right) = \psi\left(a^j u_i \right) \les \psi\left(u_{i+1}\right) \les \psi(T). \end{equation*}  This completes the proof. \end{itemize}  \end{proof} 

\begin{rmk} \rm 
The infimum in \equ{conv} is included because $\sum_{k=1}^{\infty}  m\left(B_{f, \psi\left(u_k\right), \eta, u_k}\right)^{1-r}$ may diverge for the trivial reason that there exist finitely many $k \in \N$ for which $m\left(B_{f, \psi\left(u_k\right), \eta, u_k}\right) = 0.$         
\end{rmk}    

\smallskip

\section{Applications of general results}\label{acc}

We begin with the following elementary observation.

\begin{lem}\label{local}
Let $\left(M, g_M\right)$ be an oriented $\mathscr{C}^1$ Riemannian manifold that is Hausdorff, second-countable, and without boundary. Let $\sigma$ denote the Borel measure on $M$ induced by the natural Riemannian volume form on $M$. Let $h : M \to \R^{\ell}$ be a $\mathscr{C}^1$ map, and suppose that $h^{-1}\left(\mathbf{0}_{\R^\ell}\right) \neq \varnothing.$ Let $z \in h^{-1}\left(\mathbf{0}_{\R^\ell}\right)$, and suppose that $z$ is a regular point of $h.$ Then there exist $C = C_z \in \R_{>1}$, an open subset $V = V_z$ of $\R^\ell$ with $\mathbf{0}_{\R^\ell} \in V$, and an open subset $W = W_z$ of $M$ with $z \in W$ such that for any Borel subset $E$ of $\R^\ell$ with $E \subseteq V,$ we have \begin{equation*}
C^{-1} \, m(E) \leq \sigma\big(W\cap h^{-1}(E)\big) \leq C \, m(E). 
\end{equation*}
In particular, $m|_{V}$ and the restriction to $V$ of the pushforward of $\sigma|_{W}$ by $h|_{W}$ are equivalent Borel measures.  \end{lem}

\begin{proof} For the sake of clarity, we note that $m$ denotes Lebesgue measure on $\R^\ell$. Set $k := \dim M.$ Note that $k \geq \ell.$ By the Constant Rank Theorem \cite[Theorem 7.1]{Boo} 
there exist $\varepsilon \in \R_{>0}$ and maps $\phi:(-\e,\e)^k\to M$ and $\Phi:\R^\ell \to \R^\ell$ such that:

\begin{itemize} 
\item the set $\phi\left((-\e,\e)^k\right)$ is an open subset of $M$, and $\phi$ is a $\mathscr{C}^1$ diffeomorphism onto $\phi\left((-\e,\e)^k\right)$;
\item the set $\Phi\left(\R^\ell\right)$ is an open subset of $\R^\ell$, and $\Phi$ is a $\mathscr{C}^1$ diffeomorphism onto $\Phi\left(\R^\ell\right)$;
\item $\phi\left(\mathbf{0}_{\R^k}\right) = z$; and
\item $\Phi\circ h\circ\phi = {\pi_\ell}_{|(-\e,\e)^k}$, where $\pi_\ell : \R^k\to\R^\ell$ is given by $(x_1,\dots,x_k) \mapsto (x_1,\dots,x_\ell)$.
\end{itemize}
\vspace{0.02in} 
Set $\ds W = W_z := \phi\left( \left( -\frac{\varepsilon}{2}, \frac{\varepsilon}{2} \right)^k \right) \subseteq M.$ Then $W$ is an open subset of $M$ for which \[  z \in W \subset \overline{W} = \phi\left( \left[ -\frac{\varepsilon}{2}, \frac{\varepsilon}{2} \right]^k \right) \subset \phi\left((-\e,\e)^k\right) \subseteq M. \] Note that $\Phi\left(\mathbf{0}_{\R^\ell}\right) = \mathbf{0}_{\R^\ell}.$ Set $\ds V = V_z := \Phi^{-1}\left( \left( -\frac{\varepsilon}{2}, \frac{\varepsilon}{2} \right)^\ell \right) \subseteq \R^\ell.$ Then $V$ is an open subset of $\R^\ell$ for which \[ \mathbf{0}_{\R^\ell} \in V \subset \overline{V} = \Phi^{-1}\left( \left[ -\frac{\varepsilon}{2}, \frac{\varepsilon}{2} \right]^\ell \right) \subset \Phi^{-1}\left( \left( -\varepsilon, \varepsilon \right)^\ell \right) \subseteq \R^\ell. \]     

Let $E$ be an arbitrary Borel subset of $\R^\ell$ with $E \subseteq V.$ Then \[ W\cap h^{-1}(E) = \phi\left( \left( -\frac{\varepsilon}{2}, \frac{\varepsilon}{2} \right)^k  \, \cap \,  \pi_\ell^{-1}\big(\Phi(E)\big)\right). \] Since each of $\overline{V}$ and $\overline{W}$ is compact, each of $h$ and $\pi_\ell$ is of class $\mathscr{C}^1$, and each of $\phi$ and $\Phi$ is a $\mathscr{C}^1$ diffeomorphism from its domain onto its image, it follows that \[ \sigma\big(W\cap h^{-1}(E)\big) \asymp_{{g_M, \, z}} \, m\left(  \left( -\frac{\varepsilon}{2}, \frac{\varepsilon}{2} \right)^k  \, \cap \,  \pi_\ell^{-1}\big(\Phi(E)\big)\right) = \e ^{k-\ell} \, m\big(\Phi(E)\big) \asymp_{z} \, m(E). \] \qedhere   \end{proof}

Let us now use the preceding lemma to derive a global statement. 

\begin{thm}\label{submersionvol}
Let $\left(M, g_M\right)$ be an oriented $\mathscr{C}^1$ Riemannian manifold that is compact, Hausdorff, second-countable, and without boundary. Let $\sigma$ denote the Borel measure on $M$ induced by the natural Riemannian volume form on $M$. Let $h : M \to \R^{\ell}$ be a $\mathscr{C}^1$ map; suppose that $h^{-1}\left(\mathbf{0}_{\R^\ell}\right) \neq \varnothing$ and that every element of $h^{-1}\left(\mathbf{0}_{\R^\ell}\right)$ is a regular point of $h.$ Then there exists an open subset $V$ of $\R^\ell$ with $\mathbf{0}_{\R^\ell} \in V$ such that for any Borel subset $E$ of $\R^\ell$ with $E \subseteq V,$ we have  \[ {\sigma\big(  h^{-1}(E)\big)} \asymp_{M, {g_M}, h} \, {m(E)}. \]  \end{thm}  \begin{proof} 
Set $Z := h^{-1}\left(\mathbf{0}_{\R^\ell}\right) \neq \varnothing.$ For every $z \in Z,$ let $C_z \in \R_{>1},$ $V_z \subseteq \R^\ell$ with $\mathbf{0}_{\R^\ell} \in V_z,$ and $W_z \subseteq M$ with $z \in W_z$ be as in Lemma \ref{local}. Since $Z$ is compact, there exist finitely many $z_1, \dots , z_N \in Z$ such that $\ds Z \subseteq W := \bigcup_{i=1}^N W_{z_i}. $ Set $\ds U := \bigcap_{i=1}^N V_{z_i}$ and $\ds C := \sum_{i=1}^N C_{z_i}.$ Let $V$ be an open subset of $\R^\ell$ such that $\ds \mathbf{0}_{\R^\ell} \in V \subseteq U$ and $h^{-1}(V) \subseteq W.$ (We defer the proof of the existence of $V$ until the end of this theorem's proof.) If $E$ is any Borel subset of $\R^\ell$ with $E \subseteq V,$ then \[ C^{-1} \, m(E) \le \min \big\{\sigma\big(W_{z_i}\cap h^{-1}(E)\big): 1\le i \le N\big\} \le \sigma\big(  h^{-1}(E)\big)\le  \sum_{i=1}^N \sigma\big(W_{z_i}\cap h^{-1}(E)\big)\le C \, m(E). \] 

We now prove the existence of such a set $V$; suppose by way of contradiction that such a set did not exist. This would imply that for every open subset $U'$ of $\R^\ell$ with $\ds \mathbf{0}_{\R^\ell} \in U' \subseteq U$ there exists $\yy \in (M \ssm W)$ for which $h(\yy) \in U'$. Let $\|\cdot\|$ denote the Euclidean norm on $\R^\ell.$ Then for each $r \in \N$ there exists some $\xx_r \in (M \ssm W)$ for which $h(\xx_r) \in U$ and $\| h(\xx_r) \| < r^{-1}$. Since $M$ is compact, the sequence $\ds \left( \xx_r \right)_{r \in \N}$ has a convergent subsequence whose limit we denote by $\xx \in M.$ Since $W$ is an open subset of $M$, we have $\xx \in (M \ssm W)$; this implies $\xx \notin Z$, so that $h(\xx) \neq \mathbf{0}_{\R^\ell}.$ On the other hand, the sequence $\ds \left( h(\xx_r) \right)_{r \in \N}$ clearly converges to $\mathbf{0}_{\R^\ell}$. The continuity of $h$ then implies $h(\xx) =  \mathbf{0}_{\R^\ell}.$ This is a contradiction. \qedhere  \end{proof}

\begin{standing} { Let us state here the conventions that will be in force throughout the remainder of this paper. }
\begin{itemize}
\item We shall let $G$ denote an arbitrary closed subgroup of $\operatorname{ASL}_n(\R)$, and we shall let $\mathcal{P}$ denote an arbitrary subset of $\Z^n$. We shall let $\Gamma = \Gamma(G, \mathcal{P})$ be defined as in \equ{gamma}. As usual, we shall assume that $\Gamma$ is a lattice in $G.$

\item We shall assume that $G$ is of $\mathcal{P}$-Siegel type and that we are given $r \in \R_{>1}$ for which $G$ is of $\ds \left(\mathcal{P}, {r}\right)$-Rogers type.

\item We shall let $\nu$ denote an arbitrary norm on $\R^n.$ 
\end{itemize}  
\end{standing}

\smallskip

We now state and prove Theorem \ref{submerapprox}, of which Theorem \ref{simplesubmerapprox} is an immediate consequence. \begin{thm}\label{submerapprox}
Let $\ds \psi = (\psi_1, \dots , \psi_{\ell}) : \R_{\geq 0} \to \left(\R_{>0}\right)^{\ell}$ be regular and nonincreasing. Let $\ds f = (f_1, \dots , f_{\ell}) 
: \R^n \to \R^{\ell}$ be homogeneous of degree $\dd = \dd(f) =  (d_1, \dots, d_\ell)\in \left(\R_{>0}\right)^\ell.$ Suppose further that $f$ is continuously differentiable on $\R^n_{\neq 0}$, that $\mathcal{Z}(f) \neq \varnothing$, and that each element of $\mathcal{Z}(f)$ is a regular point of $f.$ 
Let ${\pmb \xi} \in \R^\ell.$ Set $d := \sum_{j=1}^\ell d_j.$ Then the following hold. 
\begin{itemize}
\item[ \rm (i)] If $\ds \int_1^{\infty} t^{n-(d+1)} \left(\prod_{j=1}^\ell \psi_j(t) \right) \dint t$ is finite $($respectively, infinite$)$,
then $\left({_{\pmb \xi}} f\right) \circ g$ is $\left(\psi, \nu, \mathcal{P}\right)$-approximable for Haar-almost no $($respectively, almost every$)$ $g \in G$. 
\item[\rm (ii)] Suppose that $d < n$ and that the infinite series $\ds \sum_{k=1}^{\infty}\left[ 2^{k(n-d)} \prod_{j=1}^\ell \psi_j\left(2^k\right) \right]^{1-r}$ converges. Then $\left({_{\pmb \xi}} f\right) \circ g$ is {\sl uniformly} $\left(\psi, \nu, \mathcal{P}\right)$-approximable for Haar-almost every $g \in G$.
\end{itemize}
\end{thm}

\begin{proof} We begin by attending to some preliminary matters.

Let $\sigma_n$ denote the unique $\operatorname{SO}(n)$-invariant Radon probability measure on $\mathbb{S}^{n-1} \subset \R^n.$ Let $\|\cdot\|$ denote the Euclidean norm on $\R^n.$ Define $h : \mathbb{S}^{n-1} \to \R^\ell$ to be the restriction of $f$ to $\mathbb{S}^{n-1}.$ 
Note that $h^{-1}\left(\mathbf{0}_{\R^\ell}\right) = \mathcal{Z}(f) \cap \mathbb{S}^{n-1} \neq \varnothing.$ Now let $\xx \in h^{-1}\left(\mathbf{0}_{\R^\ell}\right).$ 
For each $j \in \{1, \dots , \ell \},$ the homogeneity of $f_j$ implies that $\nabla f_j (\xx)$ is tangent to $\mathbb{S}^{n-1}$. It follows that $\xx$ is a regular point of $f : \R^n \to \R^\ell$ if and only if $\xx$ is a regular point of $h : \mathbb{S}^{n-1} \to \R^\ell.$ We thereby conclude that every element of $h^{-1}\left(\mathbf{0}_{\R^\ell}\right)$ is a regular point of $h$. Theorem \ref{submersionvol} may thus be applied to $h : \mathbb{S}^{n-1} \to \R^\ell.$ Let $V \subseteq \R^\ell$ be an open neighborhood of $\mathbf{0}_{\R^\ell}$ as in the conclusion of Theorem \ref{submersionvol}. 

Let us now introduce a pair of mutually inverse bijections that we shall use in this proof:
\begin{align}
&\R_{>0} \times \mathbb{S}^{n-1} \to \R^n_{\neq 0} \quad \quad \text{given by} \quad \quad (t, \mathbf{u}) \mapsto t\mathbf{u} \label{bijectone} \\
&\text{and} \nonumber \\ 
&\R^n_{\neq 0} \to \R_{>0} \times \mathbb{S}^{n-1}\quad \quad \text{given by} \quad \quad \mathbf{x} \mapsto \left( \|\mathbf{x}\|, \frac{\mathbf{x}}{\|\mathbf{x}\|} \right). \label{bijecttwo}
\end{align}

Finally, for each $t \in \R_{>0},$ define \[ g_t := \operatorname{diag}\left( t^{-d_1}, \dots , t^{-d_\ell}\right) \in \operatorname{GL}_\ell(\R). \]

\begin{itemize}
\item[(i)] Let $s \in \R_{>0}.$ We shall show that the pair $\left({_{\pmb \xi} f}, \psi\right)$ is asymptotically acceptable by first showing that 
\eq{checkasymp}{
m\left(A_{ _{\pmb \xi} f, s\psi, \|\cdot\|}\right) < \infty \quad \text{if and only if} \quad \int_1^{\infty} t^{n-(d+1)} \left(\prod_{j=1}^\ell \psi_j(t) \right) \dint t < \infty.  
} Using the bijections in \eqref{bijectone} and \eqref{bijecttwo} and the homogeneity of $f$, it follows that for each $t \in \R_{>0}$ and each $\mathbf{u} \in \mathbb{S}^{n-1}$ we have \begin{equation}\label{sphereiff} t\mathbf{u} \in A_{ _{\pmb \xi} f, s\psi, \|\cdot\|}\quad \text{if and only if}\quad g_t \big( {\pmb \xi} - s\psi({t}) \big)  \les h(\mathbf{u}) \les g_t \big( {\pmb \xi} + s \psi({t}) \big). 
\end{equation}  Using the boundedness of $\psi$, we now fix $M \in \R_{> 2}$ such that for each $t \in \R_{\geq M}$ we have \begin{equation*}
\left\lbrace \mathbf{w} \in \R^\ell :  g_t \big( {\pmb \xi} - s\psi({t}) \big)  \les \mathbf{w} \les g_t \big( {\pmb \xi} + s \psi({t}) \big)   \right\rbrace \subseteq V . \end{equation*} Theorem \ref{submersionvol} then implies that for each $t \in \R_{\geq M}$ we have  \begin{equation}\label{followsfrom}
\sigma_n\left( \left\lbrace \mathbf{u} \in \mathbb{S}^{n-1} : g_t \big( {\pmb \xi} - s\psi({t}) \big)  \les h(\mathbf{u}) \les g_t \big( {\pmb \xi} + s \psi({t}) \big)  \right\rbrace \right) \asymp_{n, \ell, f} \ t^{-d} \, \left(\left(2 s\right)^\ell \, \prod_{j=1}^\ell \psi_j\left(t\right)\right).
\end{equation} It then follows from \eqref{sphereiff} and \eqref{followsfrom} that \begin{align*} &m\left(\{ \xx \in A_{ _{\pmb \xi} f, s\psi, \|\cdot\|} :   \|\xx\| \geq M \}\right) \\
&\asymp_n \ \int_{M}^{{\infty}} \,  t^{n-1} \  \sigma_n\left( \left\lbrace \mathbf{u} \in \mathbb{S}^{n-1} : g_t \big( {\pmb \xi} - s\psi({t}) \big)  \les h(\mathbf{u}) \les g_t \big( {\pmb \xi} + s \psi({t}) \big)  \right\rbrace \right) \dint t \\
&\asymp_{n, \ell, f} \ \left(2   s\right)^\ell \int_{M}^{{\infty}} \,  t^{n-(d+1)} \, \left( \prod_{j=1}^\ell \psi_j\left(t\right) \right) \dint t.  \end{align*} 
Since $\psi$ is bounded, this proves \equ{checkasymp}. Lemma \ref{scalingonly}(i) then implies that the pair $\left({_{\pmb \xi} f}, \psi\right)$ is asymptotically acceptable. The desired result now follows from Theorem \ref{asymptoticmain}.

\vspace{0.1in}

\item[(ii)] Let us prove that the pair $\left({_{\pmb \xi} f}, \psi\right)$ is uniformly acceptable. Let $s_1, s_2 \in \R_{>0}$ be given. Arguing as in part (i), we infer that there exists $M \in \R_{>2}$ such that for each $T \in \R$ with $T> M/s_2$, we have      
\begin{align*}
&m\left( \left\lbrace \xx \in B_{{_{\pmb \xi} f}, s_1\psi(T), \|\cdot\|, {s_2T}} :  \|\xx\| \geq M \right\rbrace \right) \\  
&\asymp_n \int_{M}^{s_2T} \,  t^{n-1} \  \sigma_n\left( \left\lbrace \mathbf{u} \in \mathbb{S}^{n-1} :  g_t\left( {\pmb \xi} - s_1\psi(T) \right)  \les f(\mathbf{u}) \les g_t\left( {\pmb \xi} + s_1\psi(T) \right)  \right\rbrace \right) \dint t \\
&\asymp_{n, \ell, f} \ \int_{M}^{s_2 T} \,  t^{n-1} \, t^{-d} \, \left( \left(2s_1\right)^\ell \, \prod_{j=1}^\ell \psi_j(T) \right) \dint t \\
&\asymp_{\ell, s_1} \ \int_{M}^{s_2 T} \,  t^{n-(d+1)} \, \left( \prod_{j=1}^\ell \psi_j(T) \right) \dint t \\  
&=  \left[\left(s_2 T \right)^{n-d} - M^{n-d}  \right] \left( \prod_{j=1}^\ell \psi_j(T) \right) \\
&\asymp_{n, d, s_2, M} \ T^{n-d} \, \prod_{j=1}^\ell \psi_j(T).   
\end{align*} It follows that  \[ 0 < \liminf_{T \to \infty} \frac{T^{n-d} \,  \prod_{j=1}^\ell \psi_j\left(T\right)}{m\left(B_{{_{\pmb \xi} f}, s_1\psi\left(T \right), \|\cdot\|, {s_2 T}} \right)} \leq \limsup_{T \to \infty} \frac{T^{n-d} \,  \prod_{j=1}^\ell \psi_j\left(T\right)}{m\left(B_{{_{\pmb \xi} f}, s_1\psi\left(T \right), \|\cdot\|, {s_2 T}} \right)} < \infty.  \] 

\vspace{0.025in}

\noindent Lemma \ref{scalingonly}(ii) now implies that the pair $\left({_{\pmb \xi} f}, \psi\right)$ is uniformly acceptable. It is clear from the preceding work that \begin{equation*}
\inf_{N \in \N}\sum_{k=N}^{\infty} m\left( B_{{_{\pmb \xi} f}, \psi\left( 2^k \right), \|\cdot\|, {2^k}} \right)^{1-r} < \infty \quad \text{if and only if} \quad \sum_{k=1}^{\infty}\left[ 2^{k(n-d)} \prod_{j=1}^\ell \psi_j\left(2^k\right) \right]^{1-r} < \infty. \end{equation*} The desired result now follows from Theorem \ref{mainuniformresult}. \end{itemize} \qedhere\end{proof}

We shall now consider examples of $f : \R^n \to \R$ that do not satisfy the nonsingularity hypotheses of Theorem \ref{submerapprox}. Since the measure estimates furnished by Theorem \ref{submersionvol} are no longer available in this setting, we shall instead use some \textit{ad hoc} measure calculations that were performed in the authors' previous paper:~see \cite[Corollaries 4.2 and 4.3]{KS}. In what follows, for each $i \in \Z_{\geq 0}$, we write $\log^i$ to denote the function $\R_{>0} \to \R$ given by $t \mapsto \left(\log t\right)^i$; in particular, $\log^0$ denotes the constant function that is equal to $1$ everywhere on $\R_{>0}.$ Our first example expands upon \cite[Corollary 4.2]{KS}; in that corollary, we considered the function $f : \R^n \to \R$  given by \eq{product}{f(\xx) := \prod_{i=1}^n \left| x_i \right| } and essentially proved the following result. 

\begin{lem}[{\cite[Corollary 4.2(i)]{KS}}]\label{lemmaprodcoord}  Let $f : \R^n \to \R$ be as in \equ{product}, let ${\overline{\psi}} : \R_{\geq 0} \to \R_{>0}$ be bounded and Borel measurable, and let $\eta$ denote the maximum norm on $\R^n.$ Then there exists $R = R \left({\overline{\psi}}, n \right) \in \R_{\geq 1}$ such that for every Borel measurable function $\psi : \R_{\geq 0} \to \R_{>0}$ with $\psi \leq {\overline{\psi}}$ on $\R_{\geq 0}$ and any real numbers $S$ and $T$ with $R \leq S \leq T,$ we have 
\begin{equation}
m\left(A_{f, \psi, \eta} \cap \{\xx \in \R^n : {S} \leq\eta(\xx) \leq {T} \} \right) = 2^n \, n \int_{S}^{T} \frac{\psi({t})}{{t}}
 \left[ \sum_{i=0}^{n-2} \frac{1}{i!} \log^i\left(\frac{{t}^n}{\psi({t})}\right)\right] \dint t. \end{equation} 
 \end{lem}  
\noindent We remark that, strictly speaking, an application of \cite[Corollary 4.2(i)]{KS} would require the function $\psi$ to be nonincreasing (see \cite[Remark 4.4(iii)]{KS}) and would provide a value of $R$ dependent on $\psi.$ That being said, an inspection of said corollary's proof shows that only the boundedness and Borel measurability of $\psi$ are required and not its monotonicity; this inspection furthermore shows that one may choose \[ R = R\left({\overline{\psi}}, n \right) := 1 + \left( \sup\left\lbrace {\overline{\psi}}(t) : t \in \R_{\geq 0 }\right\rbrace  \right)^{1/n}.  \]

\smallskip

We now consider a generalized version of the function in \equ{product}, given by raising that function to an arbitrary power $\omega \in \R_{>0}.$ 

\begin{thm}\label{productcoord}
Let $\psi : \R_{\geq 0} \to \R_{>0}$ be  regular and nonincreasing. Let $\omega \in \R_{>0}$ be arbitrary. Let $f : \R^n \to \R$ be given by $f(\xx) := \left(\prod_{i=1}^n \left| x_i \right| \right)^\omega.$ Let $\xi \in \R_{\geq 0}.$ Then the following hold.      
\vspace{0.05in}   
\begin{itemize}       
\item[(i)]  If \[ \begin{cases}
 { \ds \int_1^{\infty} \frac{\psi(t)^{1/\omega}}{t}\log^{n-2}\left(t\right) \dint t }  &  \text{ if } \xi = 0  \\ { \ds \int_1^{\infty} \frac{\psi(t)}{t} \log^{n-2}\left(t\right) \dint t }  &  \text{ if } \xi > 0 \end{cases} \] is finite $($respectively, infinite$)$,
then $\left({_{ \xi}} f\right) \circ g$ is $\left(\psi, \nu, \mathcal{P}\right)$-approximable for Haar-almost no $($respectively, almost every$)$ $g \in G$. 
\vspace{0.05in}  
\item[(ii)] If the infinite series \[ \begin{cases}
{ \ds \sum_{k=1}^{\infty}\left[ k^{n-1}  \psi\left(2^k\right) ^{1/\omega} \right]^{1-r} }   &  \text{ if } \xi = 0  \\ { \ds \sum_{k=1}^{\infty}\left[ k^{n-1} \psi\left(2^k\right) \right]^{1-r} }  &  \text{ if } \xi > 0 \end{cases} \] converges, then $\left({_{ \xi}} f\right) \circ g$ is {\sl uniformly} $\left(\psi, \nu, \mathcal{P}\right)$-approximable for Haar-almost every $g \in G$.
\end{itemize} \end{thm}  

\smallskip                

\begin{proof}
If $\ds \lim_{t \to \infty} \psi(t) > 0$, then note that we can easily construct a function $\varphi : \R_{\geq 0} \to \R_{>0}$ that is regular, nonincreasing, satisfies $\ds \lim_{t \to \infty} \varphi(t) = 0$, and for which the following hold:~when $\psi$ is replaced by $\varphi$ in Theorem \ref{productcoord}, the integral in (i) diverges and the infinite series in (ii) converges. The conclusions of (i) and (ii) will then follow for $\varphi$; we then infer that the conclusions of (i) and (ii) follow for $\psi$ as well. We therefore assume without loss of generality that $\ds \lim_{t \to \infty} \psi(t) = 0.$ In concert with the regularity and monotonicity of $\psi$, this implies that there exists some $\lambda = \lambda(\psi) \in \R_{>0}$ such that for every sufficiently large $t \in \R_{\geq 1}$, we have \begin{equation}\label{regularpsilog}  1 \leq -\log\big(\psi(t)\big) \leq \lambda\log(t). \end{equation}

Let us first discuss the case $\xi = 0$, which was the subject of \cite[Corollary 4.2]{KS}. Let us assume that $\omega = 1$, since the general case $\xi = 0$ reduces to this particular sub-case. Even though the integral and summatory conditions in that corollary and those in the $\xi = 0$ case of this theorem look different from one another, they are actually equivalent. Indeed, \eqref{regularpsilog} implies that the integral \[ \int_1^{\infty} \frac{\psi({t})}{{t}}  \log^{n-2} \left(\frac{{t}^n}{\psi({t})} \right) \dint{t}\] in \cite[Corollary 4.2(ii)]{KS} converges if and only if the integral $\ds \int_1^{\infty} \frac{\psi(t)}{t}\log^{n-2}\left(t\right) \dint t $ converges. Likewise, the infinite series in \cite[Corollary 4.2(iii)]{KS}, converges if and only if the infinite series $\ds \sum_{k=1}^{\infty}\left[ k^{n-1} \psi\left(2^k\right) \right]^{1-r}$ converges. The $\xi = 0$ case of this theorem then follows from the preceding work and \cite[Corollary 4.2]{KS}.  

\medskip

Suppose now that $\xi > 0.$ Let $\eta$ denote the maximum norm on $\R^n.$ 

\smallskip

\begin{itemize}
\item[(i)] Let $s \in \R_{>0}$ be given. Define $\overline{\psi} : \R_{\geq 0} \to \R_{>0}$ by $\overline{\psi}(t) := \left( \xi + s\psi \right)^{1/\omega}$. Let $R = R(\overline{\psi}, n)$ be as in Lemma \ref{lemmaprodcoord}. Fix $M \in \R$ for which $M > R$ and $\ds s\psi(M) < \xi/2.$ Then for each $t \in \R_{\geq M}$, we have $\ds s\psi(t) < \xi/2.$ Define $\underline{\psi} : \R_{\geq 0} \to \R_{>0}$ by \[ \underline{\psi}(t) := \begin{cases}
 { \ds \overline{\psi}(t) }  &  \text{ if } t \in [0, M)  \\ { \ds \left( \xi - s\psi(t)\right)^{1/\omega} }  &  \text{ if } t \in [M, \infty) \end{cases}. \] Let $h : \R^n \to \R$ be given by $h(\xx) := \prod_{i=1}^n \left| x_i \right|.$  Note that  \begin{equation}\label{preprodsubtract}   \left( A_{{{h}}, \overline{\psi}, \eta} \smallsetminus A_{{{h}}, \underline{\psi}, \eta}\right) \cap \eta^{-1}\left(\R_{\geq M}\right) \subseteq A_{{_{ \xi} {f}}, s\psi, \eta} \cap \eta^{-1}\left(\R_{\geq M}\right)
 \end{equation} and \begin{equation}\label{prodsubtract}  
 m\Bigg( \bigg(A_{{_{ \xi} {f}}, s\psi, \eta} \cap \eta^{-1}\left(\R_{\geq M}\right) \bigg) \smallsetminus \bigg( \left( A_{{{h}}, \overline{\psi}, \eta} \smallsetminus A_{{{h}}, \underline{\psi}, \eta}\right) \cap \eta^{-1}\left(\R_{\geq M}\right) \bigg) \Bigg) = 0.  \end{equation}  Now fix any real numbers $S$ and $T$ with $M \leq S \leq T.$ Using Lemma \ref{lemmaprodcoord}, it follows that \begin{equation}\label{prodover} m\left( A_{{{h}}, \overline{\psi}, \eta}  \cap \left\lbrace \xx \in \R^n : S \leq \eta(\xx)  \leq T \right\rbrace \right) =    2^n \, n \int_{S}^{T} \frac{\big( \xi + s\psi(t)\big)^{1/\omega}}{{t}}
 \left[ \sum_{i=0}^{n-2} \frac{1}{i!} \log^i\left(\frac{{t}^n}{\big( \xi + s\psi(t)\big)^{1/\omega}}\right)\right] \dint t
\end{equation} and  \begin{equation}\label{produnder}  m\left( A_{{{h}}, \underline{\psi}, \eta}  \cap \left\lbrace \xx \in \R^n : S \leq \eta(\xx)  \leq T \right\rbrace \right) 
 =    2^n \, n \int_{S}^{T} \frac{\big( \xi - s\psi(t)\big)^{1/\omega}}{{t}}
 \left[ \sum_{i=0}^{n-2} \frac{1}{i!} \log^i\left(\frac{{t}^n}{\big( \xi - s\psi(t)\big)^{1/\omega}}\right)\right] \dint t.  
\end{equation} 
It is easy to see that \begin{equation}\label{integralconditionone}
 m\left( A_{{_{ \xi} f}, s\psi, \eta}\right) = \infty \quad \quad \quad \quad \text{if and only if} \quad \quad \quad \quad \int_{1}^{\infty} \frac{\psi(t)}{t} \log^{n-2}\left(t\right) \dint t = \infty. 
\end{equation} This follows from \eqref{preprodsubtract} and \eqref{prodsubtract}, subtracting the right-hand side of \eqref{produnder} from that of \eqref{prodover}, using \eqref{regularpsilog} and the dominance of $\log^{n-2}$, and then performing two first-order Taylor approximations. Note that the integral criterion in \eqref{integralconditionone} is independent of $s \in \R_{>0}.$ Lemma \ref{scalingonly}(i) then implies that the pair $\left({_{ \xi} f}, \psi\right)$ is asymptotically acceptable. The desired result now follows from the foregoing work and Theorem \ref{asymptoticmain}.   

\vspace{0.05in}

\item[(ii)] Let $s_1, s_2 \in \R_{>0}$ be given. Set $J := 1 + \psi(0).$ Arguing as in part (i), we infer that there exists $M \in \R_{> 2}$ such that for each $T \in \R$ with $T > M/s_2$, we have 
\begin{align*}
m\left( \left\lbrace \xx \in B_{{_{ \xi} f}, {s_1 \psi(T)}, \eta, {s_2 T}} :  \eta(\xx) \geq M \right\rbrace \right) 
&\asymp_{n, J, M, s_1, s_2} \, \int_M^{s_2 T} \psi(T) \ t^{-1} \, \log^{n-2}(t) \, \dint t  \\
\vspace{0.025in}  
&= \psi(T) \, \left[ \log^{n-1}\left(s_2 T \right) - \log^{n-1}(M) \right] \\
\vspace{0.025in}   
&
\asymp_{n, M, s_2} \, \psi(T) \, \log^{n-1}(T). 
\end{align*} 
It follows that \[ 0 < \liminf_{T \to \infty} \frac{ \psi(T) \, \log^{n-1}(T) }{m\left(B_{{_{\xi} f}, s_1\psi(T), \eta, {s_2 T}} \right)} \leq \limsup_{T \to \infty} \frac{ \psi(T) \, \log^{n-1}(T) }{m\left(B_{{_{\xi} f}, s_1\psi(T), \eta, {s_2 T}} \right)} < \infty.  \] Lemma \ref{scalingonly}(ii) then implies that the pair $\left({_{ \xi} f}, \psi\right)$ is uniformly acceptable. It is clear from the foregoing work that \begin{equation*}
\inf_{N \in \N}\sum_{k=N}^{\infty} m\left( B_{{_{\xi} f}, \psi\left( 2^k \right), \eta, {2^k}} \right)^{1-r} < \infty \quad \text{if and only if} \quad \sum_{k=1}^{\infty}\left[ k^{n-1} \psi\left(2^k\right) \right]^{1-r} < \infty. \end{equation*} 
{An application of Theorem \ref{mainuniformresult} then yields the desired result.}
\end{itemize} 
 \end{proof}       

\begin{rmk} \rm { It is easy to see that one may modify the preceding proof to obtain a similar result when $\xi \in \R$ and $f : \R^n \to \R$ is a function of the form  $f(\xx) := \left(\prod_{i=1}^n x_i \right)^{q_1/q_2},$ where  $q_1$ and $q_2$ are arbitrary odd natural numbers. } 
\end{rmk}

Our next and final example expands upon \cite[Corollary 4.3]{KS} and is of interest because of its relation to the Khintchine\textendash Groshev
Theorem. 

\begin{thm}\label{khintchine} Let $\psi : \R_{\geq 0} \to \R_{>0}$ be  regular and nonincreasing. Let $p \in \{1, \dots , n - 1 \}$ and $\mathbf{z} = (z_1, \dots , z_p) \in \left(\R_{>0}\right)^p$ be given. Let $f: \R^n \to \R$ be given by \[ f(x_1, \dots , x_n) := \max\left\lbrace |x_i|^{z_i}: 1 \leq i \leq p\right\rbrace. \] Set $z := \sum_{i=1}^p  \left(z_i\right)^{-1}.$ Let $\xi \in \R_{\geq 0}.$ Then the following hold.    
\vspace{0.05in}     
\begin{itemize}
\item[ \rm (i)] If \[ \begin{cases}
 { \ds \int_1^{\infty} \psi(t)^z  {t}^{n - (p+1)} \ \dint t  }  &  \text{ if } \xi = 0  \\ { \ds \int_1^{\infty} \psi(t) \, {t}^{n - (p+1)} \ \dint t  }  &  \text{ if } \xi > 0 \end{cases} \] is finite $($respectively, infinite$)$,
then $\left({_{\xi}} f\right) \circ g$ is $\left(\psi, \nu, \mathcal{P} \right)$-approximable for Haar-almost no $($respectively, almost every$)$ $g \in G.$
\vspace{0.05in}
\item[\rm (ii)] If the infinite series \[ \begin{cases}
 { \ds \sum_{k = 1}^\infty 
      \left[ 2^{(n-p)k} \psi\left(2^k\right) ^z \right]^{1-r}}   &  \text{ if } \xi = 0  \\ { \ds \sum_{k = 1}^\infty 
      \left[ 2^{(n-p)k} \, \psi\left(2^k\right) \right]^{1-r}}  &  \text{ if } \xi > 0 \end{cases} \] converges, then $\left({_{\xi}} f\right) \circ g$ is {\sl uniformly} $\left(\psi, \nu, \mathcal{P} \right)$-approximable for Haar-almost every $g \in G.$ \end{itemize}  \end{thm} 

\noindent  The proof of this theorem makes use of the following lemma.

\begin{lem}[{\cite[Corollary 4.3(i)]{KS}}]\label{lemmakhintchine} Let $p \in \{1, \dots , n - 1 \}$ and $\mathbf{z} = (z_1, \dots , z_p) \in \left(\R_{>0}\right)^p$ be given. Let $f: \R^n \to \R$ be given by \[f(x_1, \dots , x_n) := \max\left\lbrace |x_i|^{z_i}: 1 \leq i \leq p\right\rbrace.\] Set $z := \sum_{i=1}^p  \left(z_i\right)^{-1}.$ Let ${\overline{\psi}} : \R_{\geq 0} \to \R_{>0}$ be bounded and Borel measurable. Let $\eta$ denote the maximum norm on $\R^n.$ Then there exists some $R = R\left({\overline{\psi}}, \mathbf{z} \right) \in \R_{\geq 1}$ such that for every Borel measurable function $\psi : \R_{\geq 0} \to \R_{>0}$ with $\psi \leq {\overline{\psi}}$ on $\R_{\geq 0}$ and any real numbers $S$ and $T$ with $R \leq S \leq T,$ we have \begin{equation*}\label{lemmaresult}
m\left(A_{f, \psi, \eta} \cap \{\xx \in \R^n : {S} \leq\eta(\xx) \leq {T} \} \right) = 2^n (n-p) \int_{S}^{T} \psi({t})^z \, t^{n-(p+1)} \, \dint t. 
\end{equation*}
\end{lem}   

\noindent  The relationship between Lemma \ref{lemmakhintchine} and \cite[Corollary 4.3(i)]{KS} is analogous to that between Lemma \ref{lemmaprodcoord} and \cite[Corollary 4.2(i)]{KS}, and similar remarks to those made earlier apply here. In particular, inspecting the proof of \cite[Corollary 4.3(i)]{KS} shows that one may choose \[  R = R\left({\overline{\psi}}, \mathbf{z} \right) := 1 + \max_{1 \leq i \leq p}  \left(\sup\left\lbrace {\overline{\psi}}(t) : t \in \R_{\geq 0 }\right\rbrace\right)^{1/{z_i}} \] in Lemma \ref{lemmakhintchine}. The $\xi = 0$ case of  Theorem \ref{khintchine} is already known:~see \cite[Corollary 4.3]{KS}. The proof of the $\xi > 0$ case of this theorem is similar to, and simpler than, that of the $\xi > 0$ case of Theorem \ref{productcoord}:~it is therefore omitted.  

\smallskip   

\begin{rmk}\label{follow} \rm 
{Let us mention that all the results in \S \ref{intro} (in particular, Theorems \ref{simplesubmerapprox}, \ref{simpleproduct}, and  \ref{simplemax} and Remarks \ref{primasl} and \ref{singularprimasl}) follow from those here in \S \ref{acc} and the fact that   $G = \operatorname{SL}_n(\R)$ and $G = \operatorname{ASL}_n(\R)$ satisfy various forms of the Siegel-type and Rogers-type axioms:~see \cite[Theorems 2.5, 2.6, and 2.8]{KS} and the references therein for details. }  \end{rmk}

\begin{rmk} \rm
We note here that one may easily deduce analogues of 
{the statements made in Remark \ref{examples}}
for Theorems \ref{productcoord}, 
\ref{khintchine}, and for the $\ell = 1$ case of Theorem \ref{submerapprox}. We also note that the $\ell = 1$ case of Theorem \ref{submerapprox} applies to the functions discussed in {Remark \ref{examples}}. \end{rmk} 

\begin{rmk} \rm
In light of the preceding results, we note that the discussion in Remark \ref{simplezerovsnonzero} applies in the current, more general, setting.    
\end{rmk}

\subsection*{Acknowledgements}
The co-author Skenderi would like to thank Jayadev Athreya and Jon Chaika for various discussions. The authors would also like to thank the anonymous referee for a detailed report whose suggestions improved the exposition in this paper.

\newpage

\end{document}